\documentclass[12pt]{amsart}

\usepackage[dvipsnames]{xcolor}
\usepackage[colorlinks,citecolor=ForestGreen,urlcolor=blue]{hyperref}
\usepackage{amsmath, amssymb, amsfonts, mathtools}
\usepackage{longtable}
\usepackage{graphicx}
\usepackage{paralist}
\usepackage{csquotes}
\usepackage{dsfont}
\usepackage{pdflscape}
\usepackage{morefloats}
\usepackage{nccfoots}
\usepackage{enumitem, calc} 
\usepackage{tikz}
\usepackage[round]{natbib}
\usepackage{booktabs}
\usepackage{xcolor}
\usepackage{multirow}
\allowdisplaybreaks

\usepackage{setspace}
\newcommand{\Var}{{\operatorname{Var}}}
\newcommand{\Cov}{{\operatorname{Cov}}}

\newcommand{\bbr}{\mathbb{R}}
\newcommand{\bbz}{\mathbb{Z}}
\newcommand{\bbn}{\mathbb{N}}
\newcommand{\bbe}{\mathbb{E}}
\newcommand{\bbp}{\mathbb{P}}
\newcommand{\ca}{\mathcal{A}}
\newcommand{\bfZ}{\mathbf{Z}}

\newcommand{\abs}[1]{\left| #1 \right|}
\newcommand{\norm}[1]{\left\| #1 \right\|}
\newcommand{\indicator}[1]{1_{\{ #1 \}}}

\newcommand{\brackets}[1]{\left( #1 \right)}
\newcommand{\floor}[1]{\lfloor #1 \rfloor}

\newcommand{\ignore}[1]{}

\theoremstyle{plain}
\newtheorem{theorem}{Theorem}[section]
\newtheorem{corollary}[theorem]{Corollary}
\newtheorem{lemma}[theorem]{Lemma}
\newtheorem{proposition}[theorem]{Proposition}
\newtheorem{assumption}[theorem]{Assumption}

\theoremstyle{definition}
\newtheorem{definition}[theorem]{Definition}

\theoremstyle{remark}
\newtheorem{remark}[theorem]{Remark}

\makeatletter
\newcommand*{\defeq}{\mathrel{\rlap{%
                     \raisebox{0.3ex}{$\m@th\cdot$}}%
                     \raisebox{-0.3ex}{$\m@th\cdot$}}%
                     =}
\makeatother

\makeatletter
\newcommand*{\eqdef}{=\mathrel{\rlap{%
                     \raisebox{0.3ex}{$\m@th\cdot$}}%
                     \raisebox{-0.3ex}{$\m@th\cdot$}}%
                     }
\makeatother

\usepackage{listings}
\begin{document}
\title[]
{Limit Theorems for the Symbolic Correlation Integral and the R\'enyi-2 Entropy under Short-range Dependence} 
\author[A. Schnurr]{Alexander Schnurr}
\author[A. Silbernagel]{Angelika Silbernagel}
\author[M. Ruiz Marin]{Manuel Ruiz Mar\'in}
\thanks{Research supported by DFG project SCHN 1231/3-2: \emph{Ordinal-Pattern-Dependence: Grenzwerts\"atze und Strukturbr\"uche im langzeitabh\"angigen Fall mit Anwendungen in Hydrologie, Medizin und Finanzmathematik}. Manuel Ruiz Marín is grateful to the support of a grant from the Spanish Ministry of Science,  Innovation and Universities PID2022-136252NB-I00 founded by MICIU/AEI/10.13039/501100011033 and by the European Regional Development Fund (FEDER, EU)}
\today

\address{A.S. and A.S., Department Mathematik, Universit\"at Siegen, 57068 Siegen, Germany}
\email{schnurr@mathematik.uni-siegen.de}

\address{M.R.M., Universidad Politécnica de Cartagena,
Calle Real 3, 30201, Cartagena, Spain}


\begin{abstract}
The symbolic correlation integral provides a way to measure the complexity of time series and dynamical systems. In the present article we prove limit results for an estimator of this quantity under the assumption of short-range dependence. Our approach is based on U-statistics. On our way we slightly generalize classical limit results in the framework of 1-approximating functionals. Furthermore, we carefully analyze the limit variance which appears in our central limit theorem.  A simulation study with ARMA and ARCH time series as well as a real world data example are also provided. In the latter we show how our method could be used to analyze EEG data in the context of epileptic seizures. 
\end{abstract}
\maketitle

\begin{center}
    This article is dedicated to the memory of Karsten Keller who initiated this research and whose ideas and comments have been a constant source of inspiration.
\end{center}

 \section{Background and Introduction}\label{backintro}

In this article we use quantities which are based on so called ordinal patterns in order to analyze the complexity of data sets and time series. 
These patterns describe the relative position of the values within a vector: Given $x=(x_1,...,x_d)\in \bbr^d$ with mutually different values, its ordinal pattern is $\pi=(\pi_1,...,\pi_d)\in \bbn^d$ where $\pi_j$ is the $j$-th rank within the values of $x$. 
\begin{figure}[h!]
    \centering
    \begin{tikzpicture}[scale=0.5]
        \draw[gray, thick] (1,3.4) -- (2,3) -- (3,2) -- (4,3.7) --(5,2.8);
        \filldraw[black] (1,3.4) circle (1.5pt);
        \filldraw[black] (2,3) circle (1.5pt);
        \filldraw[black] (3,2) circle (1.5pt);
        \filldraw[black] (4,3.7) circle (1.5pt);
        \filldraw[black] (5,2.8) circle (1.5pt);

        \draw[->] (0, 1.5) -- (0, 4.5);
        \foreach \x in {4, 5, 6}{
            \draw (0, \x-2) -- (-0.2, \x-2);
            \node at (-0.7, \x-2) {$\x$};
        };

        \draw[->, ultra thick] (6, 3) -- (8, 3);

        \node at (12, 3) {$\pi = (4, 3, 1, 5, 2)$};
    \end{tikzpicture}
    \caption{Ordinal pattern of $x = (5.3, 5.0, 4.0, 5.7, 4.8)$.}
    \label{fig: OP representations}
\end{figure}
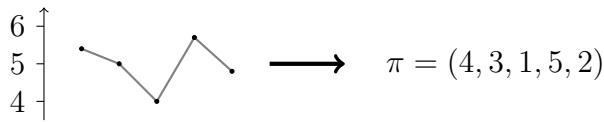

Ordinal patterns have many desirable properties like invariance under monotonic transformations, robustness with respect to small noise and simplicity in computation. They have been used in the context of dynamical systems as well as time series analysis with applications in different fields from biomedicine to econophysics (see e.g. review papers \citealp{ZaninEtAl2012}, \citealp{AmigoEtAl2015}, \citealp{sil_schn_24b}). 

The reason for introducing ordinal patterns by \cite{bandt2002_1} was to give a simple and robust measure for quantifying complexity of given data sets and the systems behind them, resulting in the concept of permutation entropy. The latter is the Shannon entropy of the ordinal pattern distribution for some $d$ underlying the data or system. For a time series  one considers ordinal pattern probabilities and for data sets one counts ordinal patterns in order to obtain relative frequencies of ordinal patterns as canonical estimates of the probabilities (under some stationarity conditions).

For large $d$, the permutation entropy is strongly related to the Kolmogorov-Sinai entropy, being a central quantity in nonlinear dynamical system theory. This fact, which was first discovered by \cite{bandt2002_1}, is important for understanding the nature of permutation entropy and might be a further reason for its success besides its good performance in data analysis. The relationship can be interpreted in such a way that long ordinal patterns contain much information of a time series and the system behind. 

Meanwhile there exist different variants of the permutation entropy (see \citealp{Keller2017}).
In the present paper we will particularly consider a variant based on R\'enyi-$q$ entropy instead of Shannon entropy: We consider (for $q \geq 0, q \neq 1$)
\begin{eqnarray*}
        R_q = R_q((p_\pi)_{\pi\in S_d})=-\log \sum_{\pi \in S_d}p_\pi^q
\end{eqnarray*}
where we sum the probabilities to the power $q$ of all possible patterns of order $d$. The quantity $R_q$ is called R\'enyi-$q$ permutation entropy (or RPE for short). 
To the best of our knowledge, RPE  has been first considered by \cite{zha_sha_hua_13} in the context of a complexity measure which puts more weight on either rare or frequent events. A first systematic study of the performance of RPE  has been conducted by \cite{liangetal_15} with regard to monitoring the depth of anaesthesia in EEG. There, the authors compared 12 entropy measures (including the classical Shannon permutation entropy as well as R\'enyi and Tsallis permutation entropy) concluding that the RPE outperformed all competitors. \cite{mammoneetal2015} investigated the RPE in the context of childhood absence epilepsy using EEG data. Their results indicate a better performance of RPE (with suitable parameters) in discriminating interictal (seizure free) from ictal (seizure) states if compared to Shannon permutation entropy. More recently, \cite{gut_kel_22} have investigated the asymptotics of RPE depending on the parameter $q$. More precisely, they were able to show that, for $q > 1$, asymptotics of RPE can be different from that of classical permutation entropy, while they coincide for $q < 1$ in the case of continuous piecewise monotone interval maps with a measure of maximal entropy. For further references on different variants of RPE proposed in the literature so far, we refer the reader to the short overview provided by \cite{gut_kel_22}.

Here, we consider the case $p=2$ which is arguably the most prominent one in the literature. 
R\'enyi-2 permutation entropy is strongly related to the symbolic correlation integral recently proposed by \cite{Caballero2019} (cf. Section 3). In fact, from a practical point-of-view the symbolic correlation integral and R\'enyi-2 permutation entropy can be used interchangeably as we will see later. The symbolic correlation integral is inspired by the widely used classical correlation integral defined by \cite{gra_pro_83}. The latter has been used in many scientific fields, for instance, physics and bioengineering; see, e.g., the references provided by \cite{acostaetal}. It is defined as the probability of two arbitrary points are within a fixed distance. Denoting this proximity parameter by $\varepsilon$, the classical correlation integral has the major disadvantage that, by construction, $\varepsilon$ has to be selected beforehand and has a high chance of affecting statistical conclusions of correlation-integral based methods. At the same time, the literature on methods for choosing the correct (range of) parameters is limited. Hence, this $\varepsilon$-dependence constitutes a non-negligible problem for practitioners (cf. \citealp{Caballero2019}).
The symbolic correlation integral, on the other hand, avoids this kind of dependence on a a priori selected parameter, and it can be understood as the degree of recurrence of ordinal patterns in a time series (cf.\ \citealp[p.~8]{Keller2017} and \citealp[p.~537]{Caballero2019}). 
\cite{wei_22} has considered the symbolic correlation integral statistic for processes consisting of independent and identically distributed real-valued random variables. He has shown that its asymptotic distribution is given by a so-called quadratic-form distribution in this case.

In this paper, our goal is to derive the limit distribution of the symbolic correlation integral (and hence also the R\'enyi-2 permutation entropy) for a broad class of short-range dependent time series, namely 1-approximating functionals. Therefore, we complement the results by \cite{wei_22}. To this end, we will take advantage of the fact that the R\'enyi-2 permutation entropy or symbolic correlation integral has a strong relation to U-statistics. 

Our contributions will prove to be useful in a variety of classification tasks, that is, they allow us to distinguish time series or data stemming from time series based on the degree of complexity present in each one of them.
For example, to a certain extent we are able to distinguish, e.g., ARMA-models that differ only in the choice of their parameters. Furthermore, our approach also allows for testing whether two data sets follow the same underlying model in the sense of a distinction between, e.g., an AR- and an MA-model.
Possible practical applications include, e.g., the distinction between healthy patients and patients with specific diseases for which the collected data exhibits a different degree of complexity if compared to the data collected from healthy patients. Well-known examples for such a diseases are epilepsy as well as different kinds of heart diseases (see \citealp{AmigoEtAl2015} and the references mentioned therein). 

The paper is organized as follows: In the subsequent two sections we describe the mathematical framework and give a brief introduction into the main quantity of interest, namely the symbolic correlation integral. Then we derive its limit distribution for the class of 1-approximating functionals (Section~\ref{section: SCI limit theorems}) and estimate the limit variance (Section~\ref{section: SCI limit variance}). This is followed by the development of hypothesis tests on whether two time series exhibit the same degree of complexity (Section~\ref{section: applications}). In this regard, we conduct a simulation study where we consider different models including AR-, MA- and ARCH-processes with varying parameters and illustrate the applicability of our test with a real-world data example. Furthermore, we compare the performance of our test with two existing tests in the literature. In Appendix~A we have collected some additional information of different entropy concepts and in Appendix~B the reader finds some auxiliary results on near-epoch dependence.

\section{Mathematical Definitions}

First, we introduce the mathematical framework. In all that follows we will work with stationary time series. The term `time series' is always meant as a discrete-time stochastic process with values in $\bbr$.

In order to derive limit theorems, we have to make some additional assumptions on the time series under consideration. Here, we will use a concept which is known as `r-approximating functional' and which is sometimes called `near-epoch dependence' in the econometric literature.  Models of this kind are very general, in particular several interesting dynamical systems are contained in this class. 

Note that there is a direct relationship between dynamical systems and stationary time series: Let $\bar{\Omega} = \bbr^\bbz$, $\bar{\mathcal{F}} = \mathcal{B}(\bbr^\bbz)$ and $\Bar{\bbp} = \bbp_X$ and define the shift-operator 
\begin{equation*}
    \tau : \Bar{\Omega} \to \Bar{\Omega}, (..., x_{-2}, x_{-1}, x_0, x_1, x_2, ...) \mapsto (..., x_{-1}, x_0, x_1, x_2, x_3, ...).
\end{equation*}
If $(X_t)_{t \in \bbz}$ is a stationary time series, then $(\Bar{\Omega}, \Bar{\mathcal{F}}, \Bar{\bbp}, \tau)$ is a measure-preserving dynamical system. 
On the other hand, let $(\Omega, \mathcal{F}, \bbp, \tau)$ denote a measure-preserving dynamical system. For any measurable function $f : \Omega \to \bbr^d$, $(X_t)_{t \in \bbz}$ defined by
\[
    X_t(\omega) = f(\tau^t (\omega))
\]
is a stationary time series (see, e.g., \citealp{Kle_20}).

In addition we will assume from now on that every finite dimensional distribution of the time series under consideration is continuous. 
As mentioned in the introduction, we will make use of the following tool in order to analyze time series (models) or data sets (reality). 

\begin{definition}
For $d\in\mathbb{N}$ let $S_d$ denote the set of permutations of $\{1, \ldots, d\}$, which we write as $d$-tuples containing each of the numbers $1, \ldots, d$ exactly once.
By the \emph{ordinal pattern of order} $d$ of some vector $x = (x_1, \ldots, x_d)$ we refer to the permutation
$\pi=(\pi_1,..., \pi_d)\in S_d$
which satisfies
\[
\pi_j < \pi_k \ \Longleftrightarrow \ x_j < x_k \ \text{ (or } x_j=x_k \text{ for } j < k)
\]
for every $j,k\in \{1,2,...,d\}$. 
\end{definition}

There are various ways to deal with ties in the definition of the patterns. From our point of view the above is easiest to handle. In the present article we assume that the distributions under consideration are continuous. This means, that the probability of ties to occur is zero. If this was not the case, one might want to consider patterns allowing for ties as in \cite{schn_fis_22b, wei_schn_23}.

\begin{definition}
Let $(Z_t)_{t\in\bbz}$ be a real-valued stationary time series. We call a sequence $(X_t)_{t\in\bbz}$ a \emph{functional of} $(Z_t)_{t\in\bbz}$ if there is a measurable function $f$ defined on $\bbr^\bbz$ such that
\begin{align}
X_t=f((Z_{t+k})_{k\in\bbz}).
\end{align}
\end{definition}

Note that $(X_t)_{t\in\bbz}$ is necessarily a stationary time series. 
The idea is now to use a quite general `background process' $(Z_t)_{t\in\bbz}$ combined with quite general class of functions $f$ in order to derive a class of processes $(X_t)_{t\in\bbz}$ which is general enough in order to contain interesting examples and on the other hand still tractable and allowing for limit theorems. 

\begin{definition}
For a time series $(Z_t)_{t\in\bbz}$ and $k \leq l$, define the $\sigma$-fields $\ca_{k}^l:=\sigma(Z_k,...,Z_l)$.
$(Z_t)_{t\in\bbz}$ is called \emph{absolutely regular} if $\beta_k\to 0$ where
\begin{align*}
\beta_k&= 2\sup_n \left\{ \sup_{A\in \ca_{n+k}^\infty} (\bbp(A|\ca_1^n)-\bbp(A)) \right\} \\
&=\sup_n \left\{ \sup \sum_{i=1}^I \sum_{j=1}^J | \bbp(A_i\cap B_j) - \bbp(A_i) \bbp(B_j) |\right\}
\end{align*}
where the last supremum is over all finite $\ca_1^n$-measurable partitions $(A_1,...,A_I)$ and all finite $\ca_{n+k}^\infty$-measurable partitions $(B_1,...,B_J)$. 
\end{definition}

This kind of time series is sometimes called `weak Bernoulli'.  On this class of time series, we consider certain functionals in order to be more general. There exists a notion of so called `Lipschitz functionals' (cf.\ \citealp[Definition~1.3]{borovkovaetal}). This is often too restrictive and has been relaxed to derive the following notion: 

\begin{definition}
Let  $(X_t)_{t\in\bbz}$ be a functional of $(Z_t)_{t\in\bbz}$ and let $r\geq 1$. Suppose that $(a_k)_{k\in\bbn_0}$ are constants with $a_k\to 0$.  We say that $(X_t)_{t\in\bbz}$ satisfies the \emph{$r$-approximating condition} or that it is an $r$\emph{-approximating functional}, if 
\begin{align} \label{approxatzero}
 \bbe\norm{X_0-\bbe(X_0|\ca_{-k}^k)}^r \leq a_k.
\end{align}
for all $k \in \bbn_0$.
The sequence $(a_k)_{k \in \bbn_0}$ of approximating constants is said to be \emph{of size} $- \lambda$ if $a_k = \mathcal{O}(k^{-\lambda-\varepsilon})$ for some $\varepsilon > 0$. 
\end{definition}

Note that the $r$-approximating condition does not necessarily require that the $r$-th moments of $X_0$ exist. One can construct simple examples, exploiting the fact, that the conditional expectation exists for every positive random variable, even if it is not integrable. 

The concepts of $L_r$-near epoch dependence and $r$-approximating functionals are closely linked, which is why the terms are sometimes used interchangeably (see, e.g., \citealp{deh_vog_wen_wie_17, schnurrdehling}): 
Under the assumption of stationarity, possible trends as considered in the definition of $L_r$-near epoch dependence are omitted, hence, $L_r$-near epoch dependence and $r$-approximating are equivalent for stationary time series. The only (notational) difference is the size of the constants under consideration: If $(Y_t)_{t \in \bbz}$ is a stationary $L_r$-near epoch dependent time series on $(Z_t)_{t \in \bbz}$ with constants $(a_k)_{k \in \bbn_0}$, then $(Y_t)_{t \in \bbz}$ is an $r$-approximating functional of $(Z_t)_{t \in \bbz}$ with constants $(a_k^r)_{k \in \bbn_0}$ (for a definition of $L_r$-near epoch dependence, see, e.g., \citealp{davidson}).

The following lemma is very important for the subsequent reasoning. As we have seen, Eq.~\eqref{approxatzero} is only demanded to hold at time zero. For stationary time series, this equation is `shift invariant'.  This fact seems to be a kind of folklore in the community. However, since we could not find a proof in the existing literature, we have decided to include it for the readers' convenience.

\begin{lemma}\label{lem: r-appox shift-invariance}
    Let $(X_t)_{t \in \bbz}$ be a (possibly $\bbr^d$-valued) functional of a stationary time series $(Z_t)_{t \in \bbz}$ and $r > 0$. Then, for all $l \geq 0$ it holds
    \begin{equation}
        \bbe\norm{X_t-\bbe(X_t|\ca_{t-l}^{t+l})}^r = \bbe\norm{X_0-\bbe(X_0|\ca_{-l}^l)}^r
    \end{equation}
    for all $t \in \bbz$.
\end{lemma}

\begin{proof}
    Define the vector $\bfZ^{l}_{k} := (Z_k, ..., Z_l)$, $k < l$, consisting of $l-k+1$ consecutive components of the time series $(Z_t)_{t \in \bbz}$. 
    Let $l \geq 0$ and $t \in \bbz$ be fixed.
    By the Doob–Dynkin lemma (or factorization lemma), there is a map $g : \bbr^{2l+1} \to \bbr^d$ such that 
    \begin{equation*}
        \bbe(X_t|\bfZ^{t+l}_{t-l}) = g(\bfZ^{t+l}_{t-l})
    \end{equation*}
    holds almost surely. Since $g$ depends only on the distribution of $(X_t, \bfZ^{t+l}_{t-l})$ (see \citealp[p.~167]{kallenberg97}) and $(X_t, \bfZ^{t+l}_{t-l}) \overset{D}{=} (X_0, \bfZ_{-l}^l)$ by stationarity and definition of $X_t$, it follows
    \begin{equation*}
        \bbe(X_t|\bfZ^{t+l}_{t-l}) = \bbe(X_0|\bfZ^l_{-l})
    \end{equation*}
    almost surely. Note that the random variables do not necessarily need to be measurable with respect to the same $\sigma$-algebra. Then, it holds
    \begin{equation*}
        \bbp\brackets{X_t-\bbe(X_t|\bfZ_{t-l}^{t+l}) \in B}
        = \bbp\brackets{X_t-\bbe(X_0|\bfZ_{-l}^l) \in B}
        = \bbp\brackets{X_0-\bbe(X_0|\bfZ_{-l}^l) \in B}
    \end{equation*}
    for all $B \in \mathcal{B}(\bbr)$, which concludes the proof.
\end{proof} 

Since we are interested in the relationship between $d$ consecutive data points, it is convenient to consider the modified time series $\overline{X}:=(\overline{X}_t)_{t\in\bbz}$ defined by the components
\[
\overline{X}_t:={(X_t,..., X_{t+d-1})'}
\]
obtained by a sliding window approach.
Here, and in all that follows, vectors in $\bbr^d$ are thought of as column vectors and we write $x'$ to denote the transposed vector. 
Note that if the time series $(X_t)_{t\in \bbz}$ is stationary, then $(\overline{X}_t)_{t\in \bbz}$ is also stationary, since considering $d$ consecutive data points is only a simple functional used on the original time series.

In addition, the following property is also preserved: 

\begin{lemma}\label{lemma: linex approx functional}
Let $r \geq 1$. 
If the time series $(X_t)_{t\in \bbz}$ can be expressed as an $r$-approximating functional $X_t=f((Z_{t+k})_{k\in \bbz})$ of the stationary time series $(Z_t)_{t\in \bbz}$ with approximating constants $(a_k)_{k \in \bbn_0}$ of size $-\lambda$, then $\overline{X}_t=g((Z_{t+k})_{k\in \bbz})$ itself satisfies the $r$-approximating condition for $k \geq d$ with approximating constants of the same size. 
\end{lemma}

\begin{proof}
Let $(X_t)_{t\in \bbz}$ be an $r$-approximating functional of  $(Z_t)_{t\in \bbz}$. It holds
\begin{align*}
    \bbe \norm{\overline{X}_0-\bbe(\overline{X}_0|\ca_{-k}^k)}_r^r
    &= \bbe \left( \abs{X_{-d+1} - \bbe(X_{-d+1}| \ca_{-k}^{k})}^r+...+ \abs{X_{0} - \bbe(X_{0}| \ca_{-k}^{k})}^r  \right)\\ 
    &= \bbe  \abs{X_{-d+1} - \bbe(X_{-d+1}| \ca_{-k}^{k})}^r+...+\bbe \abs{X_{0} - \bbe(X_{0}| \ca_{-k}^{k})}^r.
\end{align*}
Here, we have used the $r$-norm, which can be done, since every norm on $\bbr^d$ is equivalent.
For $k \geq d$, $\ca^{-i+(k-i)}_{-i-(k-i)} = \ca^{k-2i}_{-k} \subset \ca^{k}_{-k}$ constitutes a sub-$\sigma$-algebra of $\ca^{k}_{-k}$ for all $0 \leq i \leq d-1$, so, e.g., Theorem 10.28 in \cite{davidson} yields 
\begin{equation*}
    \bbe  \abs{X_{-i} - \bbe(X_{-i}| \ca_{-k}^{k})}^r \leq 2 \bbe  \abs{X_{-i} - \bbe(X_{-i}| \ca_{-k}^{k-2i})}^r
\end{equation*}
for all $0 \leq i \leq d-1$. Hence, using Lemma \ref{lem: r-appox shift-invariance} it follows 
\begin{align*}
    \bbe \norm{\overline{X}_0-\bbe(\overline{X}_0|\ca_{-k}^k)}_r^r
    &\leq 2 \sum^{d-1}_{i=0} \bbe  \abs{X_{-i} - \bbe(X_{-i}| \ca_{-k}^{k-2i})}^r \\
    &= 2 \sum^{d-1}_{i=0} \bbe  \abs{X_0 - \bbe(X_0| \ca_{-(k-i)}^{k-i})}^r \\
    &\leq 2 \sum^{d-1}_{i=0} a_{k-i}.
\end{align*}
\end{proof}

Without further assumptions we cannot validate the $r$-approximating condition of $(\overline{X}_t)_{t \in \bbz}$ for $0 \leq k < d$, that is, the boundedness of 
\begin{equation*}
    \bbe \norm{\overline{X}_0-\bbe(\overline{X}_0|\ca_{-k}^k)}_r^r \text{ for } 0 \leq k \leq d-1.
\end{equation*}
However, since $d$ is not only finite, but in particular very small in practice, with regard to limit theorems we can allow for the first $d$ approximating constants $(\overline{a}_k)_{k \in \bbn_0}$ of $(\overline{X}_t)_{t \in \bbz}$ being infinite, i.e., $\overline{a}_0 = \dots = \overline{a}_{d-1} = \infty$, because this does not have an influence on their rate of convergence. Therefore, a problem can only arise here if a summability condition is imposed on the approximating constants. This is precisely the case for the limit theorems of \cite{borovkovaetal} which we want to employ in the next section. However, the summability condition in their limit theorems can be weakened such that summability is only required for all but a finite number of approximating constants. The respective results as well as their proofs are given in Appendix~\ref{section: appendix limit theorems r-approx}.

We will make use of the following concept in Section~\ref{section: SCI limit theorems}. In a way this complements the ordinal information we usually consider with a metric information of the vector under consideration.  To our knowledge this concept is new, at least in the context of ordinal pattern analysis. 

\begin{definition}
Let $x=(x^{(1)}, ..., x^{(d)})'\in\bbr^d$. The \emph{minimal spread} of the vector is
\[
  \text{ms}(x):=\min\{\abs{x^{(j)}-x^{(k)}}: 1\leq j < k \leq d \}.
\]
\end{definition}

By our assumptions (continuous finite dimensional distributions) we always have $\bbp(\text{ms}(\overline{x}_t)=0)=0$ for every $t\in\bbz$.

\section{R\'enyi-2 Permutation Entropy and the Symbolic Correlation Integral}
\label{section: RPE and SCI}

Let $F$ denote the cumulative distribution function
\[
   F(y)=F_{\overline{X}_0}(y)=F_{\overline{X}_t}(y)=\bbp(\overline{X}_t\leq y)=\bbp(X_t \leq y_1, \dots, X_{t+d-1} \leq y_d)
\]
for $y=(y_1, \dots, y_d)\in \bbr^d$. 
Given some $t\in \bbz$ and defining $\Pi : \bbr^d \to S_d$ as the function which assigns each $d$-dimensional vector its ordinal pattern, the main quantity under consideration in this paper is 
\begin{eqnarray}
\sum_{\pi \in S_d}\bbp ( \indicator{\Pi(\overline{X}_t)=\pi})^2
&=&\sum_{\pi \in S_d}\left (\int_{\bbr^d} \indicator{\Pi(x)=\pi} \ dF(x)\right )^2 \nonumber \\
&=&\sum_{\pi \in S_d}\int_{\bbr^d} \int_{\bbr^d} \indicator{\Pi(x)=\pi}\indicator{\Pi(y)=\pi}  \ dF(x) dF(y) \nonumber \\
&=&\int_{\bbr^d} \int_{\bbr^d} \indicator{\Pi(x)=\Pi(y)} \ dF(x) dF(y) =: S^d, \label{eq: relation SCI Renyi-2 entropy}
\end{eqnarray}
which does not depend on $t \in \bbz$ by stationarity. Note that we have applied Fubini's theorem in the second equation. The r.h.s.~of \eqref{eq: relation SCI Renyi-2 entropy} is the symbolic correlation integral (SCI) as recently proposed by \cite{Caballero2019}. As the l.h.s.~shows, it is nothing more than the exponential function applied to the sign-reversed R\'enyi-2 entropy 
\begin{eqnarray*}
        R_2 = R_2((p_\pi)_{\pi\in S_d})=-\log \sum_{\pi \in S_d}p_\pi^2
\end{eqnarray*}
of the ordinal pattern distribution $(p_\pi)_{\pi\in S_d}=\bigl(\bbp (\Pi(\overline{X}_t)=\pi )\bigr)_{\pi\in S_d}$. Some information on this and similar entropies is given in Appendix~\ref{propent}. 

The definition of the SCI is highly influenced by the widespread and well-known classical correlation integral by \cite{gra_pro_83}, which is given by the probability of two arbitrary points on the orbit of the state space that are within a distance of $\varepsilon$ for some $\varepsilon > 0$ fixed beforehand (see \citealp{Caballero2019}). In other words, the correlation integral depends highly on the a priori selection of $\varepsilon$. The SCI, on the other hand, avoids this $\varepsilon$-dependence.

Eq.~\eqref{eq: relation SCI Renyi-2 entropy} shows that the SCI can be easily computed if the ordinal pattern probabilities are known. For example, in case of a uniform distribution, which is, e.g., the case for i.i.d.~series, it holds
\[
   S^d = \sum_{\pi \in S_d} p_{\pi}^2 = d! \cdot \left(\frac{1}{d!}\right)^2 = \frac{1}{d!}.
\]
For most of the time series models, nothing is known yet about the obtained ordinal pattern distributions. Exceptions are Gaussian as well as ARMA-processes (under some additional restrictions). The corresponding theorems can be found in \cite{ban_shi_07} and \cite{sin_kel_11}.

Note that the SCI attains its minimum in case of uniformly distributed ordinal patterns, that is $1/d!$, and its maximum of 1 in case of a one-point distribution. The latter corresponds to a monotonically increasing or decreasing time series. This is, in fact, the only case where a time series' ordinal patterns follow a one-point distribution as ordinal patterns are obtained in a sliding window approach. For the R\'enyi-2 permutation entropy $R_2=-\log(S^d)$, maximum and minimum are reversed.

We consider the U-statistic
\[
  S^d_n:= \frac{2}{n(n-1)}\sum_{1\leq j<k \leq n} \indicator{\Pi(\overline{X}_j)=\Pi(\overline{X}_k)}
\]
as an estimator of $S^d$. Note the connection to the recurrence of ordinal patterns by definition. Furthermore, \citet[Appendix A.1]{weis_mar_kel_mat_22} have shown that there is a strong relation between the symbolic correlation integral and the sample index of qualitative variation according to \cite{kva_95}.
By now, the limit results for this statistic are limited to the i.i.d.~case (see \citealp{wei_22}), which is a serious drawback, since several applications one has in mind exhibit serial dependence. We close this gap in the following.


\section{Limit Theorems for the Symbolic Correlation Integral}
\label{section: SCI limit theorems}

The key idea for giving more general statements is to use the deep theorems of \cite{borovkovaetal}. In order to make use of them, amongst others we have to show that our kernel
\[
  h(x,y):=\indicator{\Pi(x)=\Pi(y)}
\]
satisfies the following technical condition (cf.\ \citealp[Definition~2.12]{borovkovaetal}). We state this definition in the multivariate version. The Sections~1--5 of the article we make use of are all written down for the one-dimensional case. However, in the sixth section, the one dealing with entropy concepts, the authors state that all the previous results hold true in a $d$-dimensional setting as well. 

\begin{definition}\label{definition: p-lipschitz cond}
Let $(Y_t)_{t\in\bbz}$ be a stationary time series and $h: \bbr^{2d}\to \bbr$ be a measurable, symmetric kernel. Then we say that $h$ is \emph{$p$-continuous} if there exists a function $\phi:\left]0,\infty\right[ \to \left]0,\infty\right[$ with $\phi(\varepsilon)=o(1)$ as $\varepsilon \to 0$ such that
\begin{align}
  \bbe \abs{h(U,V)-h(\tilde{U},V)}^p \indicator{\norm{U-\tilde{U}}\leq \varepsilon} \leq \phi(\varepsilon) \label{p-lipschitz cond}
\end{align} 
for all random variables $U$, $\tilde{U}$, $V$ with marginal distribution $F=F_{Y_0}$ and such that $(U,V)$ either has distribution $F\times F$ (independent case) or $\bbp_{(Y_0,Y_t)}$ for some $t\in \bbn$. 
\end{definition}

Instead of calling a kernel $h$ $p$-continuous, some authors say that $h$ satisfies the \emph{$p$-Lipschitz condition}. 
Here, we consider $p$-continuity for our time series $(\overline{X}_t)_{t\in\bbz}$.  

\begin{proposition} \label{prop:1approx}
Let $p\geq 1$. The kernel $h:(x,y)\mapsto \indicator{\Pi(x)=\Pi(y)}$ is $p$-continuous with respect to $(\overline{X}_t)_{t\in\bbz}$.
\end{proposition}

Every norm on $\bbr^d$ is equivalent. We use the taxi-norm $\norm{\cdot}_1$ here. 

\begin{proof}
By our assumptions and continuity (from above) of the measure $\bbp$ we obtain for $t\in \bbz$
\[
  \bbp(\text{ms}(\overline{X}_t)\leq \varepsilon)=\bbp(\text{ms}(\overline{X}_1)\leq \varepsilon) \xrightarrow[\varepsilon \downarrow 0]{} \bbp(\text{ms}(\overline{X}_1)=0)=0.
\]
Now we consider with the notation from above
\begin{align}
&\bbe \abs{\indicator{\Pi(U)=\Pi(V)} - \indicator{\Pi(\tilde{U})=\Pi(V)} }^p \cdot 1_{ \{ \norm{U-\tilde{U}}_1 \leq \varepsilon \} }  \nonumber\\
&\leq \int_{\{\text{ms}(U)\leq 2\varepsilon\}}  2^p \cdot 1_{ \{ \norm{U-\tilde{U}}_1 \leq \varepsilon \} } d\bbp 
+  \int_{\{\text{ms}(U)> 2\varepsilon\}} \abs{\indicator{\Pi(U)=\Pi(V)} - \indicator{\Pi(\tilde{U})=\Pi(V)} } \cdot 1_{ \{ \norm{U-\tilde{U}}_1 \leq \varepsilon \} } d\bbp \nonumber \\
&\leq 2^p\bbp(\text{ms}(U) \leq 2\varepsilon) + 0 \label{finaleq}
\end{align}
The last inequality holds, since $\norm{U-\tilde{U}}_1 \leq \varepsilon$ implies $\abs{U^{(j)}-\tilde{U}^{(j)}}\leq \varepsilon$ for every $1\leq j \leq d$ and
since $\text{ms}(U)>2\varepsilon$, the ordinal patterns of $U$ and $\tilde{U}$ have to be the same. Hence $\indicator{\Pi(U)=\Pi(V)} - \indicator{\Pi(\tilde{U})=\Pi(V)}=0$.
The claim follows with $\phi(\varepsilon)=2^p\bbp(\text{ms}(U) \leq 2\varepsilon)$, since \eqref{finaleq} only depends on the distribution of $U$ and tends to zero for $\varepsilon\downarrow 0$. 
\end{proof}

Let us emphasize that this is a significant advantage in favor of the entropy concept we are using here. Compare in this context \cite{borovkovaetal}: They work with the Grassberger-Procaccia dimension estimator and, hence, use the family of kernels $h_t(x,y)=\indicator{\abs{x-y}\leq t}$. In order to guarantee that these satisfy the 1-continuity assumption, the family of distribution functions of $\abs{X_j-X_k}$ has to be \emph{equicontinuous} in $t$. This might be very difficult to check in practice. Using a different kernel and other method of proof, our limit results work without these equicontinuity assumptions. 

\begin{theorem}
Let $(X_t)_{t\in \bbn}$ be a 1-approximating functional of a stationary and absolutely regular time series with summable constants $(a_k)_{k \in \bbn_0}$.  Then
\[
  S^d_n \to S^d
\]
in probability, as $n$ tends to infinity. 
\end{theorem}

Note that we are still dealing with a two-sided functional, which is common practice with regard to time series analysis. But since we are considering data in order to estimate the R\'enyi-2 permutation entropy, for simplicity we omit observations indexed by $t \leq 0$.

\begin{proof}
This follows from Theorem~\ref{theorem: lln u-statistics weaker summability}, since we have shown that the time series $\overline{X}$ is stationary and itself a 1-approximating functional of an absolutely regular time series and whose approximating constants are summable disregarding the first $d$ constants (Lemma~\ref{lemma: linex approx functional}).
Furthermore, we have seen in Proposition~\ref{prop:1approx} that our kernel satisfies the 1-Lipschitz condition.  Finally,  the family of random variables $(\indicator{\Pi(X_j)=\Pi(X_k)})_{j,k\in \bbn}$ is uniformly integrable, since it is bounded by 1, which is integrable w.r.t. a probability measure.
\end{proof}

We need another function, which was introduced in \cite{borovkovaetal}. Recall that we write  $F=F_{\overline{X}_0}$. For a kernel $h:\bbr^{2d}\to\bbr$, we define
\[
  h_1(x):=\int_{\bbr^d} h(x,y) \ dF(y).
\]
Now we can prove a central limit theorem. 

\begin{theorem}\label{theorem: CLT SCI}
Let  $(X_t)_{t\in\bbn}$ be a 1-approximating functional of an absolutely regular time series with mixing coefficient $(\beta_k)_{k\in\bbn_0}$, and let $h$ and $h_1$ be as above. Suppose that the sequences $(\beta_k)_{k\in\bbn_0}$, $(a_k)_{k\in\bbn_0}$ and $(\phi(a_k))_{k\in\bbn_0}$ satisfy the following summability condition: 
\[
  \sum_{k=1}^{\infty} k^2 (\beta_k+a_k+\phi(a_k))<+\infty.
\]
Then the series
\begin{equation}
\label{eq: SCI limit variance}
    \sigma^2 = \Var(h_1(\overline{X}_1)) + 2 \sum_{k=2}^\infty \Cov(h_1(\overline{X}_1),h_1(\overline{X}_k))
\end{equation}
converges absolutely and, as $n\to\infty$,
\[
  \sqrt{n}(S_n^d-S^d)\xrightarrow[]{d} N(0,4 \sigma^2).
\]
\end{theorem}

This follows from Theorem~\ref{theorem: clt u-statistics weaker summability}, since our kernel $h$ is bounded and 1-continuous. 
Note that there is a typo in the long-run variance given in Theorem~7 of \cite{borovkovaetal} (of which Theorem~\ref{theorem: clt u-statistics weaker summability} is a generalization): In fact, the variance should not be squared. This can be verified by considering the proof of Theorem~7, which, among others, is based on Theorem~4 of \cite{borovkovaetal}. There the long-run variance with respect to centered 1-approximating functionals of absolute regular time series is given correctly.

Since in our case we know the function $\phi$ explicitly, we can derive the following:
\begin{align*}
\phi(\varepsilon)&=2\bbp(\text{ms}(\overline{X}_k) \leq 2\varepsilon) \\
&\leq 2 \sum_{k\leq i,j \leq k+d-1} \bbp(X_i-X_j \leq 2\varepsilon)
\end{align*}
This means: If the distribution functions $F_{X_i-X_j}$ (for $0\leq i,j, \leq d-1$) are all Lipschitz continuous, then the summability condition
\[
  \sum_{k=1}^{\infty} k^2 (\beta_k+a_k)<+\infty
\]
is sufficient. 

Due to $S^d > 0$ for all ordinal pattern distributions $(p_{\pi})_{\pi \in S_d}$, by application of the delta-method we directly obtain the following corollary with regard to the estimation of the Renyi-2 entropy ${\rm PE}_2 = -\ln (S^d)$:

\begin{corollary}
    Under the assumptions of Theorem \ref{theorem: CLT SCI}, it holds
    \[
        \sqrt{n} (\ln(S^d_n) - {\rm PE}_2) \xrightarrow[]{d} N(0,4 \sigma^2/(S^d)^2)
    \]
    as $n \to \infty$.
\end{corollary}

Furthermore, we can show the following interesting property of the limit variance, which we comment on subsequently. 

\begin{proposition}
\label{prop: SCI zero variance}
    Let $h$ and $h_1$ be as above.
    If the ordinal patterns of length $d$ are uniformly distributed, then it follows $\Var(h_1(\overline{X}_1)) = 0$ and hence, the limit variance $\sigma^2$ defined by \eqref{eq: SCI limit variance} equals zero.
\end{proposition}

\begin{proof}
    Since $\Var(X) = \bbe X^2 - (\bbe X)^2$ for any square-integrable random variable $X$, it suffices to show that 
    \[
        \bbe(h_1(\overline{X}_1)^2) = (S^d)^2 = \frac{1}{(d!)^2}.
    \]
    The second claim then follows by definition of the limit variance \eqref{eq: SCI limit variance} and the Cauchy-Schwarz inequality.
    
    To simplify notation, for the duration of this proof we set $\overline{X} = \overline{X}_1$.
    Now, let us consider the random variable $h_1(\overline{X})^2$. Using Fubini's theorem twice yields
    \begin{align*}
        h_1(\overline{X})^2 
        &= \brackets{\int h(\overline{X}, y) dF(y)}^2 \\
        &= \int \int h(\overline{X}, y_1) h(\overline{X}, y_2) dF(y_1) dF(y_2) \\
        &= \int \int \indicator{\Pi(\overline{X}) = \Pi(y_1) = \Pi(y_2)} dF(y_1) dF(y_2) \\
        &= \int \int \sum_{\pi \in S_d} \indicator{\Pi(\overline{X}) = \pi} \indicator{\Pi(y_1) = \Pi(y_2) = \pi} dF(y_1) dF(y_2) \\
        &= \sum_{\pi \in S_d} \indicator{\Pi(\overline{X}) = \pi} \int \int \indicator{\Pi(y_1) = \pi} \indicator{\Pi(y_2) = \pi} dF(y_1) dF(y_2) \\
        &= \sum_{\pi \in S_d} \indicator{\Pi(\overline{X}) = \pi} \brackets{\int \indicator{\Pi(y) = \pi} dF(y)}^2 \\
        &= \sum_{\pi \in S_d} \indicator{\Pi(\overline{X}) = \pi} p_{\pi}^2,
    \end{align*}
    where $p_{\pi}$ denotes the probability of the ordinal pattern $\pi \in S_d$. By application of the expected value it then follows
    \[
        \bbe(h_1(\overline{X})^2) 
        = \sum_{\pi \in S_d} p_{\pi}^2 \cdot \bbe \brackets{\indicator{\Pi(\overline{X}) = \pi}} 
        = \sum_{\pi \in S_d} p_{\pi}^3.
    \]
    Note that up to this point, we have not used our assumption, but if we do so now, then we obtain
    \[
        \bbe(h_1(\overline{X})^2) 
        = \sum_{\pi \in S_d} p_{\pi}^3 = d! \cdot \brackets{\frac{1}{d!}}^3 = \frac{1}{(d!)^2}.
    \]
\end{proof}

Hence, in the i.i.d.~case the limit variance equals 0 and thus, we end up in the theoretical limit distribution $N(0, 0)$ which denotes the point mass at the origin (see Theorem~\ref{theorem: clt u-statistics weaker summability}). 

The observation above is interesting for two reasons: Firstly, this special case of a limit distribution is not often discussed in the literature for explicit estimators - or at least we have not encountered such a case yet. It is therefore crucial to cover this case here, especially for the sake of practitioners. Secondly, this shows the value of the results of \cite{wei_22} who derived limit theorems under the assumption of i.i.d.~with regard to a different convergence rate (hence, it does not lead to a one-point distribution). Therefore, our contributions complement the results of \cite{wei_22} nicely.

Unfortunately, the author's results do not cover other models where the obtained ordinal pattern distribution is uniform, which is, for example, the case for ordinal patterns of order $d=2$ stemming from a MA(1)-process (see \citealp[Proposition~6]{ban_shi_07}). Furthermore, Proposition~\ref{prop: SCI zero variance} makes no statement about the other direction, that is, whether there exist other ordinal pattern distributions such that the limit variance $\sigma^2$ equals zero. Based on the proof, conclusions can only be drawn with regard to $\Var(h_1(\overline{X}_1))$: It is in fact zero if and only if the ordinal pattern distribution is uniform. This holds, since $\sum_{\pi \in S_d} p_{\pi}^3$ obtains its minimum only in this case. However, if $\Var(h_1(\overline{X}_1))$ is strictly positive, then it still may be canceled out by the autocovariances $\Cov(h_1(\overline{X}_1), h_1(\overline{X}_k))$, $k \geq 2$. Conditions for potential existence of such time series are still an open question.

\section{Estimating the Limit Variance}
\label{section: SCI limit variance}

If we can find an estimator $\hat{\sigma}^2_n$ such that $\hat{\sigma}^2_n \xrightarrow[]{d} \sigma^2$, then Slutsky's theorem yields $\frac{\sqrt{n}(S^d_n-S^d)}{2\hat{\sigma}_n} \xrightarrow[]{d} N(0, 1)$, so, e.g., asymptotic confidence intervals can be determined. \cite{dejongdavidson} have proposed consistent estimators based on kernels for the series on the r.h.s. of \eqref{eq: SCI limit variance}, provided some technical conditions are satisfied.

\begin{assumption}\label{assump 1}
    Let $\kappa(\cdot) : \bbr \to [-1, 1]$ be a kernel function satisfying the following conditions:
    \begin{itemize} \setlength\itemsep{1em}
        \item $\kappa(0) = 1$,
        \item $\kappa(\cdot)$ is symmetric in the sense that $\kappa(x) = \kappa(-x)$ for all $x \in \bbr$, 
        \item $\kappa(\cdot)$ is continuous at 0 and at all but a finite number of points,
        \item $\int^{\infty}_{-\infty} |\kappa(x)| dx < \infty$, and 
        \item $\int^{\infty}_{-\infty} |\psi(\xi)| d\xi < \infty$ where 
            \begin{equation*}
            \psi(\xi) = (2\pi)^{-1} \int^{\infty}_{-\infty} \kappa(x) e^{i\xi x} dx.
        \end{equation*}
    \end{itemize}
\end{assumption}

The authors stated the next assumption originally in terms of $L_r$-near epoch dependent triangular arrays. However, for our purpose consideration of time series is enough. Note that considering a time series $(Y_t)_{t \in \bbn}$ one can always define a triangular array $(Y_{nt})_{n,t \in \bbn}$ by $Y_{nt} = n^{-1/2} Y_t$. Then, using the triangular constant array $c_{nt}$ defined by $c_{nt} = n^{-1/2}$ yields Assumption~\ref{assump 2} as derived below.

\begin{assumption}\label{assump 2}
    $(Y_t)_{t \in \bbn}$ is a 2-approximating functional of size $-1$ on a strong mixing (resp. uniform mixing) time series $(Z_t)_{t \in \bbz}$ of size $-p/(p-2)$ (resp. $-p/(2p-2)$), and it holds
    \begin{equation}
        \sup_{t \geq 1} \norm{Y_t}_p < \infty \label{eq: assump 2}
    \end{equation}
    for some $p > 2$.
    Under uniform mixing, $p=2$ is also permitted if $\abs{Y_t}^2$ is uniformly integrable.
\end{assumption}

\begin{assumption}\label{assump 3}
    For a bandwidth sequence $(b_n)_{n \in \bbn}$, it holds
    \begin{equation*}
        \lim_{n \to \infty} (b_n^{-1} + b_n \cdot n^{-1}) = 0.
    \end{equation*}
\end{assumption}

There are several possible choices for the kernel function $\kappa(\cdot)$ and the bandwidth sequence $(b_n)_{n \geq 1}$ such that Assumptions \ref{assump 1} and \ref{assump 3} are fulfilled. One choice is the Bartlett-kernel $\kappa(x) = (1-|x|) \cdot 1_{[-1, 1]}(x)$ and $b_n = \log(n)$, $n \in \bbn$. 

Define
\begin{align*}
    \sigma^2_n :&= \Var(h_1(\overline{X}_1)) + 2 \sum_{k=2}^{n} \Cov(h_1(\overline{X}_1),h_1(\overline{X}_k)), \\
    &= \frac{1}{n} \sum_{i, j = 1}^{n} \Cov(h_1(\overline{X}_i),h_1(\overline{X}_j)), \\
    \hat{\sigma}^2_n :&= \frac{1}{n} \sum^{n}_{i, j=1} \kappa((j-i)/b_n) \cdot (h_1(\overline{X}_i) - S_n^d)(h_1(\overline{X}_j) - S_n^d).
\end{align*}
Note that the equality holds due to stationarity of $(\overline{X}_t)_{t \in \bbn}$. 

\begin{theorem}\label{theo: consistency sigma2}
    Let $(X_t)_{t \in \bbn}$ be a $2$-approximating functional of a stationary absolute regular time series $(Z_t)_{t \in \bbz}$ with constants $(a_k)_{k\in \bbn_0}$ and mixing coefficients $(\beta_k)_{k\in\bbn_0}$, and let $h$ and $h_1$ be as above. Suppose that the sequences $(\beta_k)_{k\in\bbn_0}$, $(a_k)_{k\in\bbn_0}$ and $(\phi(a_k))_{k\in\bbn_0}$ satisfy the following summability condition: 
    \[
    \sum_{k=1}^{\infty} k^4 (\beta_k+a_k+\phi(a_k))<+\infty.
    \] 
    Moreover, suppose $\kappa(x)$ and $(b_n)_{n \in \bbn}$ satisfy Assumptions~\ref{assump 1} and \ref{assump 3}. Then
    \begin{equation*}
        \hat{\sigma}^2_n - \sigma^2_n \xrightarrow[]{\bbp} 0.
    \end{equation*}
\end{theorem}

The proof is based on Theorem~2.1 of \cite{dejongdavidson}, hence, we need to verify Assumption~\ref{assump 2}. In this regard, we need to show that the immediate application of $h_1$ to $(\overline{X}_t)_{t\in\bbn}$ yields a $2$-approximating functional with approximating constants of the required size. For this we make use of Lemma~\ref{lemma: approx + cont} which makes a statement about the immediate application of an $r$-continuous function onto an $r$-approximating functional, and which is a generalization of Proposition~2.11 of \cite{borovkovaetal} with regard to $r$-approximating functionals in general for $r \geq 1$. For convenience, first we complement the definition of $p$-continuous functions with just one argument.

\begin{definition}
    Let $F$ be some distribution on $\bbr^d$.
    A measurable function $g : \bbr^d \to \bbr^d$ is called \emph{$p$-continuous} with respect to $F$ if there exists a function $\phi : \left]0, \infty\right[ \to \left]0, \infty\right[$ with $\phi(\varepsilon) = o(1)$ as $\varepsilon \to 0$ such that
    \begin{equation}
        \bbe \brackets{\abs{g(Y) - g(Y')}^p \indicator{\abs{Y-Y'}\leq\varepsilon}} \leq \phi(\varepsilon)
    \end{equation}
    holds for all random vectors $Y$ and $Y'$ with distribution $F$. If the underlying distribution is clearly understood, we simply say that $g$ is $p$-continuous.
\end{definition}

\begin{lemma}\label{lemma: approx + cont}
For $r \geq 1$, let $(Y_t)_{t\in\bbn}$ be an $r$-approximating functional of $(Z_t)_{t\in\bbn}$ with constants $(a_k)_{k \in \bbn_0}$ of size $-\lambda$, and let $g : \bbr^d \to \bbr^d$ be $r$-continuous with respect to the distribution of $Y_0$. Furthermore, suppose that $\norm{g(Y_0)^r}_{2+\delta}<\infty$ for some $\delta > 0$. Then $(g(Y_t))_{t\in\bbn}$ is also an $r$-approximating functional of $(Z_t)_{t\in\bbn}$ with constants
\begin{equation}
    b_k = \phi(\sqrt{2} a_k^{1/2r}) + 2^r \norm{g(Y_0)^r}_{2+\delta} \cdot (\sqrt{2}a_k^{1/2r})^{(1+\delta)/(2+\delta)}. \label{eq: approx + cont}
\end{equation}
If $g$ is bounded, then the same holds with approximating constants
\begin{equation}
    b_k = \phi(\sqrt{2} a_k^{1/2r}) + 2^r \norm{g(Y_0)^r}_{\infty} \cdot \sqrt{2} a_k^{1/2r}
\end{equation}
instead of \eqref{eq: approx + cont}. 
\end{lemma}

The proof proceeds analogously to the one given by \cite{borovkovaetal} for the restricted case of 1-approximating functionals.

\begin{proof}
    Let $(Z'_t)_{t\in\bbn}$ be a copy of $(Z_t)_{t\in\bbn}$ with $Z_t = Z'_t$ for $-k \leq t \leq k$, and let $(Y_t)_{t\in\bbn}$ and $(Y'_t)_{t\in\bbn}$ denote the respective functionals. 
    Defining $B:=\{\abs{Y_0-Y'_0}>\varepsilon\}$ it holds
    \begin{equation*}
        \bbe\abs{g(Y_0)-g(Y'_0)}^r
        = \bbe\brackets{\abs{g(Y_0)-g(Y'_0)}^r \cdot 1_{B^c}} + \bbe\brackets{\abs{g(Y_0)-g(Y'_0)}^r \cdot 1_{B}}. 
    \end{equation*}
    While we can use $\phi(\varepsilon)$ as a bound for the first summand due to the $r$-continuity of $g$, for the second summand it holds
    \begin{align*}
        \bbe\brackets{\abs{g(Y_0)-g(Y'_0)}^r \cdot 1_{B}}
        &\leq \norm{\abs{g(Y_0)-g(Y'_0)}^r}_{2+\delta} \cdot \norm{1_B}_{(2+\delta)/(1+\delta)} \\
        &\leq \norm{2^{r-1} \brackets{\abs{g(Y_0)}^r + \abs{g(Y'_0)}^r}}_{2+\delta} \cdot \bbp(B)^{(1+\delta)/(2+\delta)} \\
        &\leq 2^{r-1} \brackets{\norm{g(Y_0)^r}_{2+\delta} + \norm{g(Y'_0)^r}_{2+\delta}} \cdot \bbp(B)^{(1+\delta)/(2+\delta)} \\ 
        &= 2^r \norm{g(Y_0)^r}_{2+\delta} \cdot \bbp(B)^{(1+\delta)/(2+\delta)},
    \end{align*}
    where we applied the H\"older- and the Minkowski-inequality in the first and third line, respectively. \\
    By the $r$-approximating condition and Lemma~2.7~(i) of \cite{borovkovaetal} it holds
    \begin{equation*}
        \bbe\abs{Y_0-Y'_0}^r\leq 2^r \cdot a_k.
    \end{equation*} 
    Then, using the Markov-inequality (as well as the H\"older-inequality) yields 
    \begin{eqnarray*}
        \bbp(B) \leq \frac{\bbe \abs{Y_0-Y'_0}}{\varepsilon} \leq \frac{\brackets{\bbe \abs{Y_0-Y'_0}^r}^{1/r}}{\varepsilon} \leq \frac{2 a_k^{1/r}}{\varepsilon}.
    \end{eqnarray*}
    Choosing $\varepsilon = \sqrt{2} a_k^{1/2r}$ and summing everything up, it follows
    \begin{equation*}
        \bbe\abs{g(Y_0)-g(Y'_0)}^r \leq \phi(\sqrt{2} a_k^{1/2r}) + 2^r \norm{g(Y_0)^r}_{2+\delta} \cdot (\sqrt{2}a_k^{1/2r})^{(1+\delta)/(2+\delta)},
    \end{equation*}
    which already implies the $r$-approximating condition for $g(Y_0)$ by Lemma~2.7~(ii) of \cite{borovkovaetal}.
\end{proof}

\begin{remark}
\label{remark: approx + cont}
    The attentive reader may have noticed that by Lemma~\ref{lemma: approx + cont} it is implicitly required that the approximating constants $(a_k)_{k \in \bbn_0}$ of $(Y_t)_{t\in \bbn}$ are finite, since the function $\phi$ is defined on the open set $\left]0, \infty\right[$. However, if there is a finite number of approximating constants such that $a_k = \infty$, then without loss of generality, we may set $b_k = \infty$ for all $k$ such that $a_k = \infty$, as the approximating property is then still inherited in the same way, especially with regard to the convergence rate of the approximating constants. 
    
    Nevertheless, there exists a series of finite constants $b^\prime_k$ such that the approximating condition is satisfied: By assumption it holds $\norm{g(Y_0)^r}_{2+\delta}<\infty$, so $\norm{g(Y_0)}_r<\infty$ in particular. Without loss of generality, assume $d=1$, that is, we consider $\bbe|g(Y_0) - \bbe(g(Y_0)|\ca^{k}_{-k})|^r$. If $d>1$, we use the $r$-norm and then proceed in an analogous way. 
    It holds
    \begin{align*}
        |g(Y_0) - \bbe(g(Y_0)|\ca^{k}_{-k})|^r &\leq 2^{r-1} \left(|g(Y_0)|^r + |\bbe(g(Y_0)|\ca^{k}_{-k})|^r\right) \\
        &\leq 2^{r-1} \left(|g(Y_0)|^r + \bbe(|g(Y_0)|^r |\ca^{k}_{-k})\right) 
    \end{align*}
    for all $k \geq 0$, where we used the well-known $c_r$-inequality given by
    \begin{equation*}
        |U+V|^r \leq c_r \cdot (|U|^r + |V|^r)
    \end{equation*}
    for random variables $U$ and $V$ and 
    \begin{equation*}
        c_r := \begin{cases}
            2^{r-1} &\text{if } r>1, \\
            1 &\text{if } 0 < r \leq 1
        \end{cases}
    \end{equation*}
    (see \citealp[p.~157]{loe_77}) and the conditional Jensen-inequality. The law of total expectation then yields
     \[
        \bbe |g(Y_0) - \bbe(g(Y_0)|\ca^{k}_{-k})|^r \leq 2^{r-1} \left(\bbe |g(Y_0)|^r + \bbe\bigl(\bbe(|g(Y_0)|^r |\mathcal{A}^{k}_{-k})\bigr)\right) = 2^r \cdot \bbe |g(Y_0)|^r < \infty.
     \]
    Hence, for all $k \in \bbn_0$ there are (finite) constants which bound $\bbe\norm{g(Y_0) - \bbe(g(Y_0)|\ca^k_{-k})}^r$. 
    Let us again emphasize that the number of constants which are affected by this is finite, thus they still do not have an influence on the convergence rate as proposed by Lemma~\ref{lemma: approx + cont}.
    Therefore, even if some $a_k = \infty$, under the assumptions of Lemma~\ref{lemma: approx + cont} it follows that there is a series of (finite) approximating constants $(b^\prime_k)_{k \in \bbn_0}$ with respect to $g(Y_0)$ with the convergence rate implied by the above lemma.
\end{remark}

\begin{proof}[Proof of Theorem~\ref{theo: consistency sigma2}]
    As already mentioned, the proof is based on Theorem~2.1 of \cite{dejongdavidson}, so we need to verify Assumption~\ref{assump 2}. Without loss of generality, here it is sufficient to only consider $Y_t := h_1(\overline{X}_t)$, $t \in \bbn$, since adding a constant does not affect our results. By Lemma~\ref{lemma: linex approx functional} it holds that $(\overline{X}_t)_{t\in\bbn}$ is a 2-approximating functional of the same size as $(X_t)_{t\in\bbn}$, though $a_0, \dots, a_{d-1}$ are possibly equal to infinity. 
    Furthermore, by Proposition~\ref{prop:1approx} it follows that our kernel $h$ as defined before is 2-continuous in particular, so Lemma~2.15 of \cite{borovkovaetal} yields that $h_1$ is also 2-continuous. The 2-approximating condition is preserved when 2-continuous functions are applied (see Lemma~\ref{lemma: approx + cont}), therefore $(Y_t)_{t\in\bbn}$ remains to be 2-approximating, and even though the size of the approximating constants changes, now all approximating constants are finite (cf.\ Remark~\ref{remark: approx + cont}). \\
    Using the supposed summability conditions yields that $k^4 \beta_k \rightarrow 0$ as $k \to \infty$. Consequently, $(Z_t)_{t \in \bbz}$ is absolutely regular of size -4. Note that absolute regularity implies strong mixing of at least same size: For absolute regularity coefficients $(\beta_k)_{k\in\bbn_0}$ of size $-\lambda_0$ it holds
    \[
        k^{\lambda_0} \alpha_k \leq k^{\lambda_0} \beta_k \rightarrow 0
    \]
    as $k \to \infty$, where the series $(\alpha_k)_{k\in\bbn_0}$ denotes the strong mixing coefficients. (For further information on mixing, see, e.g., \citealp[Chapter~14]{davidson}.) Therefore, the inequality $-\frac{p}{p-2} \geq -4$ needs to be fulfilled, which is the case for $p \geq \frac{8}{3}$. Equation~\eqref{eq: assump 2} is satisfied due to the assumed stationarity: It holds
    \begin{equation*}
        \sup_{t \geq 1} \norm{Y_t}_3 
        = \norm{h_1(\overline{X}_1)}_3 \leq 1 < \infty,
    \end{equation*}
    since $h_1(x) = \bbe(1_{\Pi(x)=\Pi(\overline{X}_0)}) \leq 1$ for all $x$ by monotonicity. \\
    Furthermore, in an analogous way it follows that the approximating constants $(a_k)_{k \geq 0}$ of $(X_t)_{t \in \bbn}$ are of size -4. Denoting the approximating constants of $(\overline{X}_t)_{t \in \bbn}$ by $(\overline{a}_k)_{k\in\bbn_0}$, by Lemma~\ref{lemma: approx + cont} $(\overline{a}_k)_{k\in\bbn_0}$ is of size $-1$. 
\end{proof}


\section{Applications}
\label{section: applications}

	In this section we use the asymptotic distribution of $S_n^d$ to test  whether two time series come from the same data generating processes. Our novel test is compared with two well established tests in the literature.  Moreover, an application to real world EEG data is provided.
	
	Let $(X_t)_{t\in I}$ and $(Y_t)_{t\in I}$ be two independent time series, and denote by $S_n^d(X)$ and $S_n^d(Y)$ their respective SCIs.
	We want to test the null that
	$(X_t)_{t\in I}$ and $(Y_t)_{t\in I}$ follow the same data generating process (DGP)
	against any other alternative.
	Then, under the null it follows that $\mathbb{E}(S_n^d(X))=\mathbb{E}(S_n^d(Y))$ and, therefore, by Theorem~\ref{theorem: CLT SCI} we have that 

	\begin{equation*}\label{testdist}
		\mathcal{S}_n^d(X,Y)=\frac{\sqrt{n}(S_n^d(X)-S_n^d(Y))}{2(\sigma_n(X)+\sigma_n(Y))} \xrightarrow[]{d} N(0,1).
	\end{equation*} 
	
We have studied the new test statistic $\mathcal{S}_n^d(X,Y)$ for $d=3$ and sample sizes of $n= 2000, 5000$ and $10000$. Each process was repeated $1000$ times and the proportion of rejections of the null was calculated using a nominal size of 5\%. Various time series were generated in order to test the size and power of the new test statistic. In order to conduct size experiments the analyzed models have been the following:
\begin{description}[labelindent=1cm, leftmargin=!,labelwidth=\widthof{\bfseries DGP 4:}]
	\item[DGP 1] $X_t=0.5\epsilon_{t-1}+\epsilon_t$,\quad $Y_t=0.5\varepsilon_{t-1}+\varepsilon_t$
	\item[DGP 2] $X_t=0.5 X_{t-1}+0.8\epsilon_{t}$, \quad $Y_t=0.5 Y_{t-1}+0.8\varepsilon_{t}$
	\item[DGP 3] $X_t=0.5\vert X_{t-1}\vert^{0.8}+\epsilon_t$, \quad $Y_t=0.5\vert Y_{t-1}\vert^{0.8}+\varepsilon_t$
	\item[DGP 4] $X_t=\sqrt{h_t}\epsilon_t$, $h_t=1+0.5X_{t-1}^2$, \quad $Y_t=\sqrt{h_t}\varepsilon_t$, $h_t=1+0.5Y_{t-1}^2$
\end{description}
with  $\epsilon_t$ and $\varepsilon_t$ i.i.d. $N(0,1)$.
	
    DGP~1 is a moving average process of order 1, MA(1). DGPs~2 and 3 are linear and nonlinear AR(1) autoregressions respectively and 
    DGP~4 is a heteroscedastic ARCH(1)-process commonly employed in financial applications.
    
    Table~\ref{table: size} shows the size result.  The test is mainly conservative and reasonably well sized, with
    rejection frequencies occurring at approximately their nominal level rates except for the heteroscedastic model  (DGP~4) that exhibits rejection rates below the nominal level for the smallest sample sizes, reaching close to 0.05 for $n=10000$. 

\begin{table}[ht]
\centering
\caption{Size results of the $\mathcal{S}(X,Y)$ test.}\label{table: size}
\begin{tabular}{lcccc}
\toprule
          & DGP1 & DGP2 & DGP3 & DGP4 \\
\midrule
$n=2000$  & 0.041 & 0.041 & 0.033 & 0.013 \\
$n=5000$  & 0.045 & 0.042 & 0.049 & 0.027 \\
$n=10000$ & 0.055 & 0.050 & 0.064 & 0.054 \\
\bottomrule
\end{tabular}
\end{table}
    
In order to compute the finite sample power performance of the test, we have used $X_t$ of DGP~$i$ and $Y_t$ from DGP~$j$ for $i,j=\{1, 2, 3, 4\}$ with $i\neq j$. If $i=j$, we have to consider $X_t$ from DGP~$i$ and $Y_t$ from the same DGP as $X_t$ but changing the autoregressive parameter from 0.5 to 0.8. In this way we are not just testing for differences in the data generating process but also in the parameters that define the underlying model.
    
    Table~\ref{table: power} illustrates the results of the empirical power of the test. The power is significantly high, showing a very good performance in differentiating  the different DGPs with different dynamic properties. Looking at the power of the test when a change in the autoregressive parameter is applied, the power performance of the test is also very high, except for the heteroscedastic ARCH(1) where it shows a poor behavior. In all cases we see that the power increases with sample size.
    
\begin{table}[ht]
\centering
\caption{Power results of the $\mathcal{S}(X,Y)$ test. Comparing a DGP with itself, the autoregressive parameter was changed.}\label{table: power}
\begin{tabular}{llcccc}
\toprule
          &            & DGP1 & DGP2 & DGP3 & DGP4 \\
\midrule
\multirow{3}{*}{DGP1} 
          & $n=2000$   & 0.961 & 0.490 & 0.976 & 1     \\
          & $n=5000$   & 1     & 0.886 & 1     & 1     \\
          & $n=10000$  & 1     & 0.992 & 1     & 1     \\
\addlinespace
\multirow{3}{*}{DGP2} 
          & $n=2000$   & --    & 0.894 & 0.538 & 0.996 \\
          & $n=5000$   & --    & 1     & 0.888 & 1     \\
          & $n=10000$  & --    & 1     & 0.995 & 1     \\
\addlinespace
\multirow{3}{*}{DGP3} 
          & $n=2000$   & --    & --    & 0.387 & 0.389 \\
          & $n=5000$   & --    & --    & 0.790 & 0.931 \\
          & $n=10000$  & --    & --    & 0.976 & 1     \\
\addlinespace
\multirow{3}{*}{DGP4} 
          & $n=2000$   & --    & --    & --    & 0.030 \\
          & $n=5000$   & --    & --    & --    & 0.097 \\
          & $n=10000$  & --    & --    & --    & 0.217 \\
\bottomrule
\end{tabular}
\end{table}

The $\mathcal{S}(X,Y)$ test is compared with two tests, one based on the empirical distribution functions, the classical Kolmogorov-Smirnov two sample test (KS-test) (see \citealp{smirnov1939estimation}), and a more recent one (JP-test) based on the spectral densities of the time series (see \citealp{jentsch2015testing}). 

The two-sample KS-test is a non-parametric test that evaluates whether two independent samples come from the same continuous distribution. The test compares their empirical distribution functions (EDFs) with the statistic:
$$KS_{n,m} = \sup_x \left| F_n(x) - G_m(x) \right|$$
which is the maximum absolute difference between the two EDFs.
Under the null hypothesis and as \( n, m \to \infty \), the scaled statistic converges in distribution to the supremum of a Brownian bridge:
$$\sqrt{\frac{nm}{n + m}} KS_{n,m} \xrightarrow{d} \sup_{0 \leq t \leq 1} |B(t)|$$
where \( B(t) \) is a Brownian bridge process. The critical values are obtained from the Kolmogorov distribution.

On the other hand, the JP-statistic tests whether two (or more) stationary processes have the same spectral density:
$$
    H_0: f_X(\omega) = f_Y(\omega) \quad \text{for all } \omega \in [-\pi, \pi].
$$
Each spectral density is estimated using a kernel estimator:
$$
    \hat{f}_{j,h}(\omega_k) = \frac{1}{n} \sum_{\ell=1}^{n} K_h(\omega_k - \omega_\ell) I_j(\omega_\ell)
$$
where $I_j(\omega_\ell)$ is the periodogram of series $j$ at frequency $\omega_\ell$, and $K_h$ is a smoothing kernel with bandwidth $h$.
The test statistic is:
$$ 
    JP_n = \frac{1}{n h^{1/2}} \sum_{k=1}^n \left( \hat{f}_{X,h}(\omega_k) - \hat{f}_{Y,h}(\omega_k) \right)^2
$$
that under the null
$$JP_n \xrightarrow{d} \mathcal{N}(0, \tau^2)$$
where $\tau^2$ is a long-run variance.
Nevertheless, as acknowledged by the authors, statistics of this type converge very slowly to their asymptotic normal distribution, which results in a poor performance of the corresponding asymptotic test in terms of size. To address this issue, the authors propose a suitable randomization technique that approximates the finite-sample distribution of the JP-test statistic using data-driven critical values. In what follows, we adopt the randomized version of the statistic, as recommended by the authors.

For both the Kolmogorov–Smirnov (KS) and Jentsch–Pauly (JP) tests, we simulate the same data-generating processes described earlier. Specifically, we consider a sample size of $n=2000$, and for each configuration, we perform 1000 Monte Carlo replications to assess the empirical  behavior of the test statistics at a nominal level of 5\%.

\begin{table}[ht]
\centering
\caption{Size of KS and JP tests.}\label{table: sizeKSJP}
\begin{tabular}{lcccc}
\toprule
         & DGP1 & DGP2 & DGP3 & DGP4 \\ 
\midrule
KS-test  & 0.124 & 0.221 & 0.093 & 0.057 \\ 
JP-test  & 0.110 & 0.072 & 0.137 & 0.677 \\ 
\bottomrule
\end{tabular}
\end{table}
The empirical size figures reported in Table \ref{table: sizeKSJP} reveal non‑negligible distortions for both goodness‑of‑fit procedures.
Under all four DGPs, the KS-test rejects the null much more often than the nominal 5\% rate would prescribe. The problem is particularly acute for DGP2, where the estimated size climbs to 22.1\%, but it remains roughly twice the nominal level for DGPs 1 and 3 and only approaches correct calibration under DGP4 (5.7\%). The JP-test exhibits an even more pronounced liberal behaviour. Its size reaches 13.7\% for DGP3 and explodes to 67.7\% for DGP4, indicating a severe inflation of Type‑I error probability for heteroscedastic processes.

This systematic over‑rejection naturally translates into the performance displayed in Table~\ref{table: powerKSJP}. Because both tests already reject too frequently under the null, their empirical power appears artificially high across the alternative scenarios. Consequently, the ostensibly strong power should be interpreted with caution: It is driven in large part by the size distortions documented in Table \ref{table: sizeKSJP}, rather than by genuine discriminatory capability. 

\begin{table}[ht]
\centering
\caption{Power of KS and JP tests.}\label{table: powerKSJP}
\begin{tabular}{llcccc}
\toprule
      &        & DGP1 & DGP2 & DGP3 & DGP4 \\
\midrule
\multirow{2}{*}{DGP1} 
      & KS-test & 1      & 0.227  & 1      & 0.663 \\ 
      & JP-test & 1      & 1     & 0.967  & 1      \\[4pt]

\multirow{2}{*}{DGP2} 
      & KS-test & --      & 1 & 1      & 0.4430 \\ 
      & JP-test & --      & 1     & 0.942  & 1      \\[4pt]

\multirow{2}{*}{DGP3} 
      & KS-test & --      & --     & 1      & 1      \\ 
      & JP-test & --      & --     & 1      & 1      \\[4pt]

\multirow{2}{*}{DGP4} 
      & KS-test & --      & --     & --      & 0.778 \\ 
      & JP-test & --      & --     & --      & 1      \\
\bottomrule
\end{tabular}
\end{table}

In summary, the $\mathcal{S}(X,Y)$ test, built on the symbolic correlation integral, offers practical advantages over classical EDF tests such as Kolmogorov-Smirnov (KS) and $L^2$-type spectral density tests like the JP-test. First, no smoothing window, proximity threshold $\varepsilon$, or bandwidth has to be tuned, eliminating the size distortions that plague KS (when serial dependence is present) and the spectral test (whose performance hinges on kernel-bandwidth selection). Second, $\mathcal{S}(X,Y)$ captures a higher-order non-linear structure. KS compares only marginal empirical distributions, and the spectral test targets second-order (covariance) properties. $\mathcal{S}(X,Y)$, by contrast, counts pattern recurrences, making it sensitive to subtle dynamical changes for which a central limit theorem has been proved via U-statistics. Third, the absence of kernels or block resampling keeps variance and computational load small, producing reliable p-values even for a few thousand points where KS or spectral tests often missize.

 To show the power performance of the new test on real complex data, we have analyzed sets of electroencephalographic (EEG) time series. Concretely we have used surface EEG recordings from healthy volunteers with eyes closed ($O$) and eyes open ($Z$), and intracranial EEG recordings from epilepsy patients during the seizure free interval from within ($F$) and from outside ($N$) the seizure generating area as well as intracranial EEG recordings of epileptic seizures ($S$) (see \citealp{andrzejak2001indications}, for a more detailed description of the dataset).

\begin{table}[h!]
\caption{Values of $\mathcal{S}(X,Y)$ statistic and the respective p-values for all possible combinations of the five EEG time series.\label{table: EEG}}
\begin{tabular}{cccccc}
\cline{2-6}
\multicolumn{1}{l|}{}                    & \multicolumn{1}{c|}{$\mathcal{S}(S,F)$} & \multicolumn{1}{c|}{$\mathcal{S}(S,N)$}  & \multicolumn{1}{c|}{$\mathcal{S}(S,O)$} & \multicolumn{1}{c|}{$\mathcal{S}(S,Z)$} & \multicolumn{1}{c|}{$\mathcal{S}(F,N)$} \\ \hline
\multicolumn{1}{|c|}{\textbf{Statistic}} & \multicolumn{1}{c|}{41.826}             & \multicolumn{1}{c|}{52.554}             & \multicolumn{1}{c|}{30.178}             & \multicolumn{1}{c|}{51.829}             & \multicolumn{1}{c|}{10.729}             \\ \hline
\multicolumn{1}{|c|}{\textbf{p-value}}   & \multicolumn{1}{c|}{$< 0.001$}             & \multicolumn{1}{c|}{$< 0.001$}             & \multicolumn{1}{c|}{$<0.001$}             & \multicolumn{1}{c|}{$<0.001$}             & \multicolumn{1}{c|}{$<0.001$}             \\ \hline
\multicolumn{1}{l}{}                     & \multicolumn{1}{l}{}                    & \multicolumn{1}{l}{}                    & \multicolumn{1}{l}{}                    & \multicolumn{1}{l}{}                    & \multicolumn{1}{l}{}                    \\ \cline{2-6} 
\multicolumn{1}{l|}{}                    & \multicolumn{1}{c|}{$\mathcal{S}(F,O)$} & \multicolumn{1}{c|}{$\mathcal{S}(F,Z)$} & \multicolumn{1}{c|}{$\mathcal{S}(N,O)$} & \multicolumn{1}{c|}{$\mathcal{S}(N,Z)$} & \multicolumn{1}{c|}{$\mathcal{S}(O,Z)$}   \\ \hline
\multicolumn{1}{|c|}{\textbf{Statistic}} & \multicolumn{1}{c|}{-11.648}            & \multicolumn{1}{c|}{10.003}             & \multicolumn{1}{c|}{-22.377}            & \multicolumn{1}{c|}{-0.726}             & \multicolumn{1}{c|}{21.651}             \\ \hline
\multicolumn{1}{|c|}{\textbf{p-value}}   & \multicolumn{1}{c|}{$<0.001$}             & \multicolumn{1}{c|}{$<0.001$}             & \multicolumn{1}{c|}{$<0.001$}             & \multicolumn{1}{c|}{0.035}              & \multicolumn{1}{c|}{$<0.001$}             \\ \hline
\end{tabular}
\end{table}

 Table~\ref{table: EEG} shows the results of the $\mathcal{S}(X,Y)$ test for all pairwise possibilities of the five EEG time series described above. In all cases the test statistic rejects the null hypothesis that the time series come from the same data generating process, at a significance level of $0.05$.

\textbf{Declaration:} The authors report there are no competing interests to declare.


\bibliography{references}

\begin{thebibliography}{}

\bibitem[Acosta et~al., 2024]{acostaetal}
Acosta, J., Vallejos, R., and G{\'o}mez, J. (2024).
\newblock Correlation integral for stationary gaussian time series.
\newblock {\em Sankhya A}, 86(1):191--214.

\bibitem[Amig\`{o} et~al., 2015]{AmigoEtAl2015}
Amig\`{o}, J.~M., Keller, K., and Unakafova, V.~A. (2015).
\newblock Ordinal symbolic analysis and its application to biomedical recordings.
\newblock {\em Philosophical transactions. Series A, Mathematical, physical, and engineering sciences}, 373(2034):20140091.

\bibitem[Andrzejak et~al., 2001]{andrzejak2001indications}
Andrzejak, R.~G., Lehnertz, K., Mormann, F., Rieke, C., David, P., and Elger, C.~E. (2001).
\newblock Indications of nonlinear deterministic and finite-dimensional structures in time series of brain electrical activity: Dependence on recording region and brain state.
\newblock {\em Physical Review E}, 64(6):061907.

\bibitem[Bandt and Pompe, 2002]{bandt2002_1}
Bandt, C. and Pompe, B. (2002).
\newblock Permutation entropy: A natural complexity measure for time series.
\newblock {\em Phys. Rev. Lett.}, 88:174102.

\bibitem[Bandt and Shiha, 2007]{ban_shi_07}
Bandt, C. and Shiha, F. (2007).
\newblock Order patterns in time series.
\newblock {\em Journal of Time Series Analysis}, 28(5):646--665.

\bibitem[Borovkova et~al., 2001]{borovkovaetal}
Borovkova, S., Burton, R., and Dehling, H. (2001).
\newblock Limit theorems for functionals of mixing processes with applications to {U}-statistics and dimension estimation.
\newblock {\em Transactions of the American Mathematical Society}, 353(11):4261--4318.

\bibitem[Caballero-Pintado et~al., 2019]{Caballero2019}
Caballero-Pintado, M.~V., Matilla-Garc\'{i}a, M., and Mar\'{i}n, M.~R. (2019).
\newblock Symbolic correlation integral.
\newblock {\em Econometric Reviews}, 38(5):533--556.

\bibitem[Davidson, 1994]{davidson}
Davidson, J. (1994).
\newblock {\em Stochastic limit theory: An introduction for econometricians}.
\newblock OUP Oxford.

\bibitem[De~Jong and Davidson, 2000]{dejongdavidson}
De~Jong, R.~M. and Davidson, J. (2000).
\newblock Consistency of kernel estimators of heteroscedastic and autocorrelated covariance matrices.
\newblock {\em Econometrica}, 68(2):407--423.

\bibitem[Dehling et~al., 2017]{deh_vog_wen_wie_17}
Dehling, H., Vogel, D., Wendler, M., and Wied, D. (2017).
\newblock Testing for changes in {K}endall’s tau.
\newblock {\em Econometric Theory}, 33(6):1352–1386.

\bibitem[Dehling and Wendler, 2010]{deh_wen_10}
Dehling, H. and Wendler, M. (2010).
\newblock Central limit theorem and the bootstrap for u-statistics of strongly mixing data.
\newblock {\em Journal of Multivariate Analysis}, 101(1):126--137.

\bibitem[Grassberger and Procaccia, 1983]{gra_pro_83}
Grassberger, P. and Procaccia, I. (1983).
\newblock Measuring the strangeness of strange attractors.
\newblock {\em Physica D: nonlinear phenomena}, 9(1-2):189--208.

\bibitem[Gutjahr and Keller, 2022]{gut_kel_22}
Gutjahr, T. and Keller, K. (2022).
\newblock On r{\'e}nyi permutation entropy.
\newblock {\em Entropy}, 24(1).

\bibitem[J\"ackle and Keller, 2017]{axioms6020014}
J\"ackle, S. and Keller, K. (2017).
\newblock Tsallis entropy and generalized {S}hannon additivity.
\newblock {\em Axioms}, 6(2).

\bibitem[Jentsch and Pauly, 2015]{jentsch2015testing}
Jentsch, C. and Pauly, M. (2015).
\newblock Testing equality of spectral densities using randomization techniques.
\newblock {\em Bernoulli}, 21(2):697--739.

\bibitem[Kallenberg, 2021]{kallenberg97}
Kallenberg, O. (2021).
\newblock {\em Foundations of modern probability}, volume~99 of {\em Probability Theory and Stochastic Modelling}.
\newblock Springer Cham, 3 edition.

\bibitem[Keller et~al., 2017]{Keller2017}
Keller, K., Mangold, T., Stolz, I., and Werner, J. (2017).
\newblock Permutation entropy: New ideas and challenges.
\newblock {\em Entropy}, 19(3).

\bibitem[Klenke, 2020]{Kle_20}
Klenke, A. (2020).
\newblock {\em Ergodic Theory}, pages 493--513.
\newblock Springer International Publishing.

\bibitem[Kvålseth, 1995]{kva_95}
Kvålseth, T.~O. (1995).
\newblock Coefficients of variation for nominal and ordinal categorical data.
\newblock {\em Perceptual and Motor Skills}, 80(3):843--847.

\bibitem[Leucht and Neumann, 2013a]{leu_neu_13}
Leucht, A. and Neumann, M.~H. (2013a).
\newblock Degenerate {\$}{\$}u{\$}{\$}- and {\$}{\$}v{\$}{\$}-statistics under ergodicity: asymptotics, bootstrap and applications in statistics.
\newblock {\em Annals of the Institute of Statistical Mathematics}, 65(2):349--386.

\bibitem[Leucht and Neumann, 2013b]{leu_neu_13b}
Leucht, A. and Neumann, M.~H. (2013b).
\newblock Dependent wild bootstrap for degenerate u- and v-statistics.
\newblock {\em Journal of Multivariate Analysis}, 117:257--280.

\bibitem[Liang et~al., 2015]{liangetal_15}
Liang, Z., Wang, Y., Sun, X., Li, D., Voss, L.~J., Sleigh, J.~W., Hagihira, S., and Li, X. (2015).
\newblock Eeg entropy measures in anesthesia.
\newblock {\em Frontiers in computational neuroscience}, 9:16.

\bibitem[Lo\`eve, 1977]{loe_77}
Lo\`eve, M. (1977).
\newblock {\em Probability Theory I}.
\newblock Graduate Texts in Mathematics 45. Springer New York, New York, NY, 4th ed. 1977 edition.

\bibitem[Mammone et~al., 2015]{mammoneetal2015}
Mammone, N., Duun-Henriksen, J., Kjaer, T.~W., and Morabito, F.~C. (2015).
\newblock Differentiating interictal and ictal states in childhood absence epilepsy through permutation r{\'e}nyi entropy.
\newblock {\em Entropy}, 17(7):4627--4643.

\bibitem[Schnurr and Dehling, 2017]{schnurrdehling}
Schnurr, A. and Dehling, H. (2017).
\newblock Testing for structural breaks via ordinal pattern dependence.
\newblock {\em Journal of the American Statistical Association}, 112(518):706--720.

\bibitem[Schnurr and Fischer, 2022]{schn_fis_22b}
Schnurr, A. and Fischer, S. (2022).
\newblock Generalized ordinal patterns allowing for ties and their applications in hydrology.
\newblock {\em Computational Statistics \& Data Analysis}, 171:107472.

\bibitem[Silbernagel and Schnurr, 2024]{sil_schn_24b}
Silbernagel, A. and Schnurr, A. (2024).
\newblock An overview of various applications of ordinal patterns in data analysis and mathematical statistics.
\newblock In {\em 2024 International Conference on Electrical, Computer and Energy Technologies (ICECET)}, pages 1--5. IEEE.

\bibitem[Sinn and Keller, 2011]{sin_kel_11}
Sinn, M. and Keller, K. (2011).
\newblock Estimation of ordinal pattern probabilities in {G}aussian processes with stationary increments.
\newblock {\em Computational Statistics \& Data Analysis}, 55(4):1781--1790.

\bibitem[Smirnov, 1939]{smirnov1939estimation}
Smirnov, N.~V. (1939).
\newblock On the estimation of the discrepancy between empirical curves of distribution for two independent samples.
\newblock {\em Bull. Math. Univ. Moscou}, 2(2):3--14.

\bibitem[Wei{\ss}, 2022]{wei_22}
Wei{\ss}, C.~H. (2022).
\newblock Non-parametric tests for serial dependence in time series based on asymptotic implementations of ordinal-pattern statistics.
\newblock {\em Chaos: An Interdisciplinary Journal of Nonlinear Science}, 32(9):093107.

\bibitem[Wei{\ss} et~al., 2022]{weis_mar_kel_mat_22}
Wei{\ss}, C.~H., Mar{\'\i}n, M.~R., Keller, K., and Matilla-Garc{\'\i}a, M. (2022).
\newblock Non-parametric analysis of serial dependence in time series using ordinal patterns.
\newblock {\em Computational Statistics \& Data Analysis}, 168:107381.

\bibitem[Wei{\ss} and Schnurr, 2023]{wei_schn_23}
Wei{\ss}, C.~H. and Schnurr, A. (2023).
\newblock Generalized ordinal patterns in discrete-valued time series: Non-parametric testing for serial dependence.
\newblock {\em Journal of Nonparametric Statistics}, pages 1--27.

\bibitem[Zanin et~al., 2012]{ZaninEtAl2012}
Zanin, M., Zunino, L., Rosso, O., and Papo, D. (2012).
\newblock Permutation entropy and its main biomedical and econophysics applications: A review.
\newblock {\em Entropy}, 14:1553.

\bibitem[Zhao et~al., 2013]{zha_sha_hua_13}
Zhao, X., Shang, P., and Huang, J. (2013).
\newblock Permutation complexity and dependence measures of time series.
\newblock {\em Europhysics Letters}, 102(4):40005.

\end{thebibliography}
\bibliographystyle{apalike}
\appendix
\section{Properties of Entropies}\label{propent}
We consider an entropy $H$ as a map from the set $\Delta$ of all stochastic vectors $(p_i)_{i=1}^l;l\in {\mathbb N}$  ($p_i\in [0,1]$ for $i=1,2,\ldots ,l$ and $\sum_{i=1}^l p_i=1$) into $\left[0,\infty \right[$.  Here $H(p_1,p_2,\ldots , p_l)$ stands for $H((p_i)_{i=1}^l)$.

Besides the classical Shannon entropy defined by
\begin{align*}
H^S((p_i)_{i=1}^l)=-\sum_{i=1}^{l}p_i\ln p_i
\end{align*}
for $(p_i)_{i=1}^l\in\Delta$,
the R\'enyi entropy and the Tsallis entropy
defined by
\begin{align*}
H^R_\alpha((p_i)_{i=1}^l)=\frac{\ln\sum_{i=1}^{n}p_i^\alpha}{1-\alpha}
\end{align*}
and
\begin{align*}
H^T_\alpha((p_i)_{i=1}^l)=\frac{1-\sum_{i=1}^{n}p_i^\alpha}{\alpha -1},
\end{align*}
respectively, for $(p_i)_{i=1}^l\in\Delta$ and given $\alpha\in\ \left]0,\infty\right[\,\setminus\,\{1\}$ are of some special interest in many fields. The Shannon entropy itself is considered as both the Tsallis entropy and the Renyi entropy for $\alpha=1$.  This is justified by the fact that
\begin{align*}
\lim_{\alpha\to\infty} H^T_\alpha(p_1,p_2,\ldots ,p_l)=\lim_{\alpha\to\infty} H^R_\alpha(p_1,p_2,\ldots ,p_l)=H^S_\alpha(p_1,p_2,\ldots ,p_l)
\end{align*}
for all $(p_i)_{i=1}^l\in\Delta$.
Also note that,  given some $\alpha\in\ \left]0,\infty\right[\setminus\{1\}$,  it holds
\begin{align*}
H^R_\alpha(p_1,p_2,\ldots ,p_l)=\frac{\ln (1-(\alpha -1)H^T_\alpha(p_1,p_2,\ldots ,p_l))}{1-\alpha}
\end{align*}
for all $(p_i)_{i=1}^l\in\Delta$, implying that
\begin{align}\label{monot}
H^R_\alpha(p_1,p_2,\ldots ,p_l)\leq H^R_\alpha(q_1,q_2,\ldots ,q_{l'})\Longleftrightarrow H^T_\alpha(p_1,p_2,\ldots ,p_l)\leq H^T_\alpha(q_1,q_2,\ldots ,q_{l'})\nonumber\\\mbox{ for all }(p_i)_{i=1}^l, (q_i)_{i=1}^{l'}\in\Delta .
\end{align}

It is well known that,  for each $\alpha\in\ ]0,\infty[$,  the Tsallis entropy satisfies the  meaning that
\begin{eqnarray*}
&&H_{\alpha}^T(p_1,\ldots ,p_{k-1},p_k,\ldots,p_m,p_{m+1},\ldots ,p_l)\nonumber\\
&&\hspace{0.25cm}=H_{\alpha}^T\left (p_1,\ldots ,p_{k-1},\sum_{i=k}^m p_i,p_{m+1},\ldots ,p_l\right )+\left (\sum_{i=k}^m p_i\right  )^{\hspace{-1mm}\alpha } H_{\alpha}^T\left (\frac{p_k}{\sum_{i=k}^m p_i},\ldots ,\frac{p_m}{\sum_{i=k}^m p_i}\right)\nonumber\\
&&\hspace{0.5cm}\mbox{for all }(p_1,\ldots ,p_{k-1},p_k,\ldots,p_m,p_{m+1},\ldots ,p_l)\in \triangle;\, k,m,l\in {\mathbb N};\, k\leq m\leq l.
\end{eqnarray*}
For more details,  in particular,  further references and a characterization of the Tsallis entropy based on the generalized Shannon additivity, see \citealp{axioms6020014}.

Since $H_{\alpha}^T(p_1,p_2,\ldots ,p_l)$ for fixed $n\in {\mathbb N}$ and $(p_1,p_2,\ldots ,p_l)\in\Delta$ takes its only maximum value at the point $\overbrace{\left (\frac{1}{l},\ldots ,\frac{1}{l}\right )}^{n\mbox{\scriptsize \ times}}$,  it follows
\begin{eqnarray*}
&&H_{\alpha}^T(p_1,\ldots ,p_{k-1},p_k,\ldots,p_m,p_{m+1},\ldots ,p_l)\nonumber\\
&&\hspace{0.25cm}\leq H_{\alpha}^T(p_1,\ldots ,p_{k-1},
\overbrace{\frac{\sum_{i=k}^m p_i}{m-k+1},\ldots ,\frac{\sum_{i=k}^m p_i}{m-k+1}}^{m-k+1\mbox{\scriptsize \ times}},p_{m+1},\ldots ,p_l)\nonumber\\
&&\hspace{0.5cm}\mbox{for all }(p_1,\ldots ,p_{k-1},p_k,\ldots,p_m,p_{m+1},\ldots ,p_l)\in \triangle;\, k,m,l\in {\mathbb N};\,k\leq m\leq l
\end{eqnarray*}
with equality for $p_m=p_{m+1}=\ldots =p_{k-1}=p_k$.  By \eqref{monot},  the analogue statement is valid for the R\'enyi entropy 
$H_{\alpha}^R$. 

\section{Near-epoch Dependence and U-Statistics: Auxiliary Results} \label{section: appendix limit theorems r-approx}

\cite{borovkovaetal} have generalized the theory regarding limit theorems for U-statistics on short-range-dependent time series significantly; for an overview on some earlier results on limit theorems for U-statistics we refer the reader in particular to \cite[p. 4295, 4298]{borovkovaetal}. Some more recent contributions include \cite{deh_wen_10} and \cite{deh_vog_wen_wie_17}, as well as \cite{leu_neu_13} and \cite{leu_neu_13b} who have considered the asymptotic distribution of degenerate U-statistics.

\cite{borovkovaetal} have proposed their theory with regard to $r$-approximating functionals whose approximating constants satisfy some sort of summability criterion. We weaken the assumptions given in this theory such that we only require summability up to a finite number of constants. Even though our proofs follow mainly those of the authors, we include them here for the sake of completeness.

\subsection{Near regularity}

In order to consider the main theorems, which are given by a law of large numbers and a central limit theorem for U-statistics when the underlying time series is an $r$-approximating functional of an absolutely regular time series, first we consider some of the preceding results of \cite{borovkovaetal}, which the authors used for the proofs of their main results, with regard to a generalized summability condition.

\begin{definition}
    \begin{enumerate}
        \item Let $(X_t)_{t\in\bbn}$ be a time series, let $M,N \in \bbn$ be positive integers and moreover assume that $M$ is even. An \emph{$(M,N)$-blocking of $(X_t)_{t\in\bbn}$} is defined as the sequence of blocks $B_1, B_2, \dots$ each consisting of $N$ consecutive $X_t$'s, where each two consecutive blocks are separated by blocks of length $M$, more precisely, 
        \begin{equation*}
            B_s = (X_{(s-1)(M+N)+\frac{M}{2}+1}, \dots, X_{s(M+N)-\frac{M}{2}}), \: s \geq 1.
        \end{equation*}
        The sets of indices in a block $B_s$ are denoted by $I_s := \{(s-1)(M+N)+\frac{M}{2}+1, \dots, s(M+N)-\frac{M}{2}\}$.
        \item A time series $(X_t)_{t\in\bbn}$ is called \emph{nearly regular} if for any $\varepsilon, \delta > 0$ there exists an $M \in \bbn$ such that for all $N \in \bbn$ we can find a sequence $(B'_s)_{s\in\bbn}$ consisting of independent $\bbr^N$-valued random vectors which satisfy the following two conditions:
        \begin{enumerate}
            \item Let $B_s$ denote the $s$-th $(M, N)$-block of $(X_t)_{t\in\bbn}$. Then, $B'_s$ has the same distribution as $B_s$.
            \item It holds $\bbp(\norm{B_s-B'_s}\leq\delta) \geq 1 - \varepsilon$, where $\norm{\cdot}$ denotes the $L_1$-norm on $\bbr^N$, that is, $\norm{x} = \sum^N_{i=1} \abs{x_i}$.
        \end{enumerate}
    \end{enumerate}
\end{definition}

``Absolute regularity of a process implies that the sequence of $(M,N)$-blocks can be perfectly coupled with the sequence of independent long blocks, which have the same distribution as those of the original process''
\citep[p.~4271]{borovkovaetal}. Near regularity in turn refers to the similarity to absolute regularity in terms of this property, that is, near regularity implies closeness to such a time series. 

Theorem~3 of \cite{borovkovaetal} shows that 1-approximating functionals with summable approximating constants are nearly regular. However, since the authors have not used the summability condition in their proof, their result remains valid even if the summability condition is not satisfied at all, though their result then becomes trivial. For our purpose it is enough to consider the following generalization:

\begin{theorem}\label{theorem: Theorem 3}
    Let $(X_t)_{t\in\bbn}$ be a 1-approximating functional with approximating constants $(a_k)_{k \in \bbn_0}$ of an absolutely regular time series with mixing coefficients $(\beta_k)_{k\in\bbn_0}$. If there are integers $K, L, N \in \bbn$, $K$ even, such that $\sum^{\infty}_{k=L} a_k < \infty$, then we can approximate the sequence of $(K+2L, N)$-blocks $(B_s)_{s\in\bbn}$ by a sequence of independent blocks $(B'_s)_{s\in\bbn}$ with the same marginal distribution in such a way that 
    \begin{equation}
        \bbp(\norm{B_s-B'_s} \leq 2a_L) \geq 1 - \beta_K - 2a_L,
    \end{equation}
    where
    \begin{equation}
        a_L := (2\sum^{\infty}_{k=L} a_l)^{1/2},
    \end{equation}
    so $(X_t)_{t\in\bbn}$ is nearly regular by definition.
\end{theorem}

\subsection{Moment inequalities and central limit theorem for partial sums}

Now we generalize some inequalities for second and fourth order moments of partial sums $S_n = X_1 + \dots + X_n$ of 1-approximating functionals of absolutely regular time series which have been given by \cite{borovkovaetal}. Our proofs are closely linked to the original proofs, but our basic idea is to split the respective sums appearing in the proofs according to our altered summability condition. Hence, the difference is that we need to find a different bound for the first summands.

\begin{lemma}\label{lemma: Lemma 2.23}
Let $(X_t)_{t\in\bbn}$ be a 1-approximating functional with constants $(a_k)_{k \in \bbn_0}$ of an absolutely regular time series with mixing coefficients $(\beta_k)_{k\in\bbn_0}$. Moreover, suppose that $\bbe X_0 = 0$ and that one of the following conditions holds for a fixed integer $d \geq 0$:
\begin{enumerate}
    \item $X_0$ is bounded a.s.\ and $\sum^{\infty}_{k=d} (a_k + \beta_k) < \infty$.
    \item $\bbe |X_0|^{2+\delta} < \infty$ for some $\delta > 0$ and $\sum^{\infty}_{k=d} (a_k^{\frac{\delta}{1+\delta}} + \beta_k^{\frac{\delta}{2+\delta}}) < \infty$.
\end{enumerate}
Then it follows
\begin{equation}
    \frac{1}{n} \bbe S_n^2 \rightarrow \bbe X_0^2 + 2 \sum^{\infty}_{k=1} \bbe(X_0 X_k)
\end{equation}
as $n \to \infty$ and the sum on the r.h.s.\ converges absolutely.
\end{lemma}

\begin{proof}
    The proofs for the different conditions are basically the same using the respective results of Lemma 2.18 of \cite{borovkovaetal}. Therefore we just give a proof under the second condition. \\
    It holds 
    \begin{align*}
        \bbe S_n^2 &= \sum_{1 \leq i, j \leq n} \bbe X_i X_j \\
        &= n \bbe X_0^2 + 2 \sum^n_{k=1} (n-k) \bbe X_0 X_k \\
        &= n \biggl(\bbe X_0^2 + 2 \Bigl(\sum^{3d-1}_{k=1} \bigl(1 - \frac{k}{n}\bigr) \bbe X_0 X_k + \sum^n_{k=3d} (1 - \frac{k}{n}) \bbe X_0 X_k\Bigr)\biggr)
    \end{align*}
    by stationarity of $(X_t)_{t\in\bbn}$. \\
    Due to $\bbe |X_0|^{2+\delta} < \infty$, $\bbe X_0^2$ as well as $\bbe X_0 X_k$ are bounded by a constant $C_1 > 0$ for $1 \leq k < 3d$, so in particular, $\sum^{3d-1}_{k=1} (1 - \frac{k}{n}) \bbe X_0 X_k$ is bounded by $(3d-1)C_1$. Note the difference to \cite{borovkovaetal} at this point. Since $\bbe X_0 = 0$ by assumption, for $k \geq 3d$ Lemma 2.18 (ii) in that paper yields 
    \begin{equation*}
        |\bbe X_0 X_k| \leq 4 \norm{X_0}^{\frac{\delta}{1+\delta}}_{2+\delta} (a_{\floor{\frac{k}{3}}})^{\frac{\delta}{1+\delta}} + 2 \norm{X_0}^2_{2+\delta} (\beta_{\floor{\frac{k}{3}}})^{\frac{\delta}{2+\delta}}.
    \end{equation*}
    Hence, 
    \begin{equation*}
        \sum^{\infty}_{k=3d} |\bbe X_0 X_k|
        \leq C_2 \sum^{\infty}_{k=3d} \Bigl((a_{\floor{\frac{k}{3}}})^{\frac{\delta}{1+\delta}} + (\beta_{\floor{\frac{k}{3}}})^{\frac{\delta}{2+\delta}}\Bigr) 
        = 3C_2 \sum^{\infty}_{k=d} \Bigl(a_{k}^{\frac{\delta}{1+\delta}} + \beta_{k}^{\frac{\delta}{2+\delta}}\Bigr)
    \end{equation*}
    converges absolutely by assumption, where $C_2 := \max\{4 \norm{X_0}^{\frac{\delta}{1+\delta}}_{2+\delta}, 2 \norm{X_0}^2_{2+\delta}\}$. Then the dominated convergence theorem yields 
    \begin{equation*}
        \sum^n_{k=3d} \bigl(1 - \frac{k}{n}\bigr) \bbe X_0 X_k \rightarrow \sum^{\infty}_{k=3d} \bbe X_0 X_k
    \end{equation*}
    as $n \to \infty$, which proves our claim.
\end{proof}

\begin{lemma}\label{lemma: Lemma 2.24}
Let $(X_t)_{t\in\bbn}$ be a 1-approximating functional with constants $(a_k)_{k \in \bbn_0}$ of an absolutely regular time series with mixing coefficients $(\beta_k)_{k\in\bbn_0}$. Moreover, suppose that $\bbe X_0 = 0$ and that one of the following conditions holds for a fixed integer $d \geq 0$:
\begin{enumerate}
    \item $X_0$ is bounded a.s.\ and $\sum^{\infty}_{k=d} k^2 (a_k + \beta_k) < \infty$.
    \item $\bbe |X_0|^{4+\delta} < \infty$ for some $\delta > 0$ and $\sum^{\infty}_{k=d} k^2 (a_k^{\frac{\delta}{3+\delta}} + \beta_k^{\frac{\delta}{4+\delta}}) < \infty$.
\end{enumerate}
Then there exists a constant $C > 0$ such that
\begin{equation}
    \bbe S_n^4 \leq C n^2.
\end{equation}
\end{lemma}

The proof follows the same idea, but is slightly more involved.

\begin{proof}
    Again, the proofs are very similar using the respective results of Lemma 2.18 as well as Lemma 2.21 and Lemma 2.22 in \cite{borovkovaetal}, respectively, so we only give the proof with regard to the second condition. By stationarity it holds
    \begin{align}
        \bbe S_n^4 
        &\leq 4! \sum_{1 \leq i_1 \leq i_2 \leq i_3 \leq i_4 \leq n} \abs{\bbe (X_{i_1} X_{i_2} X_{i_3} X_{i_4})} \nonumber \\
        &\leq 4! n \sum_{\substack{i, j, k \geq 0 \\ i+j+k \leq n}} \abs{\bbe (X_{0} X_{i} X_{i+j} X_{i+j+k})}. \label{eq: lemma 2.24 sum}
    \end{align}
    We can split the sum appearing in \eqref{eq: lemma 2.24 sum} in the following way:
    \begin{align*}
        &\sum_{\substack{i, j, k \geq 0 \\ i+j+k \leq n}} \abs{\bbe (X_{0} X_{i} X_{i+j} X_{i+j+k})} \\
        &\hspace{15mm} \leq \sum_{0 \leq j,k \leq i \leq n} \abs{\bbe (X_{0} X_{i} X_{i+j} X_{i+j+k})} 
        + \sum_{0 \leq i,j \leq k \leq n} \abs{\bbe (X_{0} X_{i} X_{i+j} X_{i+j+k})} \\
        &\hspace{30mm} + \sum_{0 \leq i,k \leq j \leq n} \abs{\bbe (X_{0} X_{i} X_{i+j} X_{i+j+k})} \\
        &\hspace{15mm} = \underbrace{\sum_{0 \leq j,k \leq i < 3d} \abs{\bbe (X_{0} X_{i} X_{i+j} X_{i+j+k})}}_{\text{(I)}} \ + \sum_{\substack{3d \leq i \leq n \\ 0 \leq j,k \leq i}} \abs{\bbe (X_{0} X_{i} X_{i+j} X_{i+j+k})} \\
        &\hspace{30mm} + \underbrace{\sum_{0 \leq i,j \leq k < 3d} \abs{\bbe (X_{0} X_{i} X_{i+j} X_{i+j+k})}}_{\text{(II)}} \ + \sum_{\substack{3d \leq k \leq n \\ 0 \leq i, j \leq k}} \abs{\bbe (X_{0} X_{i} X_{i+j} X_{i+j+k})} \\
        &\hspace{30mm} + \underbrace{\sum_{\substack{0 \leq i,k \leq j < 3d\\\phantom{0k}\\\phantom{3k}}} \abs{\bbe (X_{0} X_{i} X_{i+j} X_{i+j+k})}}_{\text{(III)}} \ +
        \underbrace{\sum_{\substack{3d \leq j \leq n \\ 0 \leq i,k \leq j \\ i < 3d \text{ or } k < 3d}} \abs{\bbe (X_{0} X_{i} X_{i+j} X_{i+j+k})}}_{\text{(IV)}} \\
        &\hspace{45mm} + \underbrace{\sum_{3d \leq i,k \leq j \leq n} \abs{\bbe (X_{0} X_{i} X_{i+j} X_{i+j+k})}}_{\text{(V)}}.
    \end{align*}
    Due to $\bbe |X_0|^{4+\delta} < \infty$ and the assumed stationarity, there exists a constant $C_1 > 0$ such that $\abs{\bbe (X_{0} X_{i} X_{i+j} X_{i+j+k})} < C_1$ (as well as a constant $C_2 >0$ such that $\abs{\bbe (X_{0} X_{i})} < C_2$), and hence, the finite sums (I), (II) and (III) which are independent of $n$ are obviously finite, too. 
    In particular it follows 
    \begin{equation*}
        \sum_{0 \leq j,k \leq i < 3d} \abs{\bbe (X_{0} X_{i} X_{i+j} X_{i+j+k})} \leq (3d)^3 \cdot C_1 = 27d^3 \cdot C_1
    \end{equation*}
    for sums of this type. By adding a zero-valued term we obtain
    \begin{align*}
        \abs{\bbe (X_{0} X_{i} X_{i+j} X_{i+j+k})} 
        &= \abs{\bbe (X_{0} X_{i} X_{i+j} X_{i+j+k})} - \abs{\bbe(X_0 X_i)} \cdot \abs{\bbe(X_{i+j}X_{i+j+k})} \\
        &\hspace{20mm} + \abs{\bbe(X_0 X_i)} \cdot \abs{\bbe(X_{i+j}X_{i+j+k})}
    \end{align*}
    such that we can find the following bound for (IV) + (V):
    \begin{align*}
        &\sum_{\substack{3d \leq j \leq n \\ 0 \leq i,k \leq j \\ i < 3d \text{ or } k < 3d}} \abs{\bbe (X_{0} X_{i} X_{i+j} X_{i+j+k})} +\sum_{3d \leq i,k \leq j \leq n} \abs{\bbe (X_{0} X_{i} X_{i+j} X_{i+j+k})} \\
        &\hspace{15mm} =\sum_{\substack{3d \leq j \leq n \\ 0 \leq i,k \leq j \\ i < 3d \text{ or } k < 3d}} \abs{\bbe (X_{0} X_{i} X_{i+j} X_{i+j+k})} - \abs{\bbe(X_0 X_i)} \cdot \abs{\bbe(X_{i+j}X_{i+j+k})} \\
        &\hspace{30mm} +  \sum_{\substack{3d \leq j \leq n \\ 0 \leq i,k \leq j \\ i < 3d \text{ or } k < 3d}} \abs{\bbe(X_0 X_i)} \cdot \abs{\bbe(X_{i+j}X_{i+j+k})} \\
        &\hspace{30mm} + \sum_{3d \leq i,k \leq j \leq n} \abs{\bbe (X_{0} X_{i} X_{i+j} X_{i+j+k})} - \abs{\bbe(X_0 X_i)} \cdot \abs{\bbe(X_{i+j}X_{i+j+k})} \\
        &\hspace{30mm} + \sum_{3d \leq i,k \leq j \leq n} \abs{\bbe(X_0 X_i)} \cdot \abs{\bbe(X_{i+j}X_{i+j+k})} \\
        &\hspace{15mm} \leq \sum_{\substack{3d \leq j \leq n \\ 0 \leq i,k \leq j}} \abs{\bbe (X_{0} X_{i} X_{i+j} X_{i+j+k})} - \abs{\bbe(X_0 X_i)} \cdot \abs{\bbe(X_{i+j}X_{i+j+k})} \\
        & \hspace{30mm} + n \cdot \sum_{3d \leq i,k \leq n} \abs{\bbe(X_0 X_i)} \cdot \abs{\bbe(X_{0}X_{k})} + n \cdot \sum_{\substack{0 \leq i,k \leq n \\ i < 3d \text{ or } k < 3d}} \abs{\bbe(X_0 X_i)} \cdot \abs{\bbe(X_{0}X_{k})}.
    \end{align*}
    Note that there we have summed up the first and the third term and have bounded the second and fourth term by using stationarity. Accordingly it follows
    \begin{align}
        &\sum_{\substack{i, j, k \geq 0 \\ i+j+k \leq n}} \abs{\bbe (X_{0} X_{i} X_{i+j} X_{i+j+k})} \nonumber \\ 
        &\hspace{10mm} \leq \ 3 \cdot 27d^3 \cdot C_1 + \sum_{\substack{3d \leq i \leq n \\ 0 \leq j,k \leq i}} \abs{\bbe (X_{0} X_{i} X_{i+j} X_{i+j+k})} 
        \ + \sum_{\substack{3d \leq k \leq n \\ 0 \leq i,j \leq k}} \abs{\bbe (X_{0} X_{i} X_{i+j} X_{i+j+k})} \nonumber \\
        & \hspace{20mm} + \sum_{\substack{3d \leq j \leq n \\ 0 \leq i,k \leq j}} \abs{\bbe (X_{0} X_{i} X_{i+j} X_{i+j+k})} - \abs{\bbe(X_0 X_i)} \cdot \abs{\bbe(X_{i+j}X_{i+j+k})} \nonumber \\
        & \hspace{20mm} + \ n \cdot \sum_{3d \leq i,k \leq n} \abs{\bbe(X_0 X_i)} \cdot \abs{\bbe(X_0 X_k)} \ + \ n \cdot \sum_{\substack{0 \leq i,k \leq n \\ i < 3d \text{ or } k < 3d}} \abs{\bbe(X_0 X_i)} \cdot \abs{\bbe(X_0 X_k)}. \label{eq: lemma 2.24 interim ineq}
    \end{align}
    The idea is now to apply different upper bounds on the remaining summands. In more detail: Since $\bbe X_0 = 0$, Lemma~2.22 of \cite{borovkovaetal} implies 
    \begin{align*}
        \abs{\bbe (X_{0} (X_{i} X_{i+j} X_{i+j+k}))} 
        &\leq 6(\beta_{\floor{\frac{i}{3}}})^{\frac{\delta}{4+\delta}}\norm{X_0}^4_{4+\delta}+8(a_{\floor{\frac{i}{3}}})^{\frac{\delta}{3+\delta}}\norm{X_0}^{\frac{12+3\delta}{3+\delta}}_{4+\delta}, \\
        \abs{\bbe ((X_{0} X_{i} X_{i+j}) X_{i+j+k})}
        &\leq 6(\beta_{\floor{\frac{k}{3}}})^{\frac{\delta}{4+\delta}}\norm{X_0}^4_{4+\delta}+8(a_{\floor{\frac{k}{3}}})^{\frac{\delta}{3+\delta}}\norm{X_0}^{\frac{12+3\delta}{3+\delta}}_{4+\delta}.
    \end{align*}
    Furthermore, application of Lemma~2.22 and Lemma~2.18 of \cite{borovkovaetal}, respectively, yields
    \begin{align*}
        &\abs{\bbe ((X_{0} X_{i}) (X_{i+j} X_{i+j+k}))} - \abs{\bbe(X_0 X_i)} \cdot \abs{\bbe(X_{i+j}X_{i+j+k})} \\
        & \hspace{30mm} \leq 6(\beta_{\floor{\frac{j}{3}}})^{\frac{\delta}{4+\delta}}\norm{X_0}^4_{4+\delta}+8(a_{\floor{\frac{j}{3}}})^{\frac{\delta}{3+\delta}}\norm{X_0}^{\frac{12+3\delta}{3+\delta}}_{4+\delta}
    \end{align*}
    and
    \begin{align*}
        \abs{\bbe(X_0 X_i)} \cdot \abs{\bbe(X_{0}X_{k})} 
        &\leq \brackets{2\norm{X_0}^2_{4+\delta}(\beta_{\floor{\frac{i}{3}}})^{\frac{2+\delta}{4+\delta}} + 4\norm{X_0}^{\frac{4+\delta}{3+\delta}}_{4+\delta} (a_{\floor{\frac{i}{3}}})^{\frac{2+\delta}{3+\delta}}} \\
        &\hspace{15mm} \times \brackets{2\norm{X_0}^2_{4+\delta}(\beta_{\floor{\frac{k}{3}}})^{\frac{2+\delta}{4+\delta}} + 4\norm{X_0}^{\frac{4+\delta}{3+\delta}}_{4+\delta} (a_{\floor{\frac{k}{3}}})^{\frac{2+\delta}{3+\delta}}}.
    \end{align*}
    Hence, we obtain 
    \begin{align*}
        &\sum_{\substack{0 \leq i,k \leq n \\ i < 3d \text{ or } k < 3d}} \abs{\bbe(X_0 X_i)} \cdot \abs{\bbe(X_0 X_k)} \\
        &\hspace{15mm} \leq \sum_{\phantom{i<}0 \leq i,k \leq 3d} \abs{\bbe(X_0 X_i)} \cdot \abs{\bbe(X_0 X_k)} + 2 \cdot \sum_{\substack{0 \leq i < 3d \\ 3d \leq k \leq n}} \abs{\bbe(X_0 X_i)} \cdot \abs{\bbe(X_0 X_k)} \\
        &\hspace{15mm} \leq (3d)^2 C_2^2 + 2 \cdot 3d \cdot C_2 \sum_{k=3d}^n \abs{\bbe(X_0 X_k)} \\
        &\hspace{15mm} \leq 9d^2 C_2^2 + 6d \cdot C_2 \sum_{k=3d}^n \brackets{2\norm{X_0}^2_{4+\delta}(\beta_{\floor{\frac{k}{3}}})^{\frac{2+\delta}{4+\delta}} + 4\norm{X_0}^{\frac{4+\delta}{3+\delta}}_{4+\delta} (a_{\floor{\frac{k}{3}}})^{\frac{2+\delta}{3+\delta}}}
    \end{align*}
    with regard to the last summand appearing in \eqref{eq: lemma 2.24 interim ineq}. (Note that we have used Lemma~2.18 of \cite{borovkovaetal} for the last inequality.) 
    Therefore, there is a constant $C_3 > 0$ such that 
    \begin{align*}
        \bbe S_n^4 
        &\leq (4! \cdot 81d^3 \cdot C_1)n + (4! \cdot 9d^2 C_2^2) n^2
        + C_3 n \Biggl( \sum_{\substack{3d \leq i \leq n \\ 0 \leq j,k \leq i}} \brackets{(a_{\floor{\frac{i}{3}}})^{\frac{\delta}{3+\delta}} + (\beta_{\floor{\frac{i}{3}}})^{\frac{\delta}{4+\delta}}} \\
        &\hspace{15mm} + \sum_{\substack{3d \leq k \leq n \\ 0 \leq i,j \leq k}} \brackets{(a_{\floor{\frac{k}{3}}})^{\frac{\delta}{3+\delta}} + (\beta_{\floor{\frac{k}{3}}})^{\frac{\delta}{4+\delta}}} + \sum_{\substack{3d \leq j \leq n \\ 0 \leq i,k \leq j}} \brackets{(a_{\floor{\frac{j}{3}}})^{\frac{\delta}{3+\delta}} + (\beta_{\floor{\frac{j}{3}}})^{\frac{\delta}{4+\delta}}} \\
        &\hspace{15mm} + n \cdot \sum_{3d \leq i,k \leq n} \brackets{(\beta_{\floor{\frac{i}{3}}})^{\frac{2+\delta}{4+\delta}} + (a_{\floor{\frac{i}{3}}})^{\frac{2+\delta}{3+\delta}}} \brackets{(\beta_{\floor{\frac{k}{3}}})^{\frac{2+\delta}{4+\delta}} + (a_{\floor{\frac{k}{3}}})^{\frac{2+\delta}{3+\delta}}} \\
        &\hspace{15mm} + n \sum_{k=3d}^n \brackets{(\beta_{\floor{\frac{k}{3}}})^{\frac{2+\delta}{4+\delta}} + (a_{\floor{\frac{k}{3}}})^{\frac{2+\delta}{3+\delta}}}\Biggr) \\
        &\leq (4! \cdot 81d^3 \cdot C_1)n 
        + (4! \cdot 9d^2 C_2^2) n^2 
        + C_3 n \Biggl( 3 \sum_{\substack{3d \leq j \leq n \\ 0 \leq i,k \leq j}} \brackets{(a_{\floor{\frac{j}{3}}})^{\frac{\delta}{3+\delta}} + (\beta_{\floor{\frac{j}{3}}})^{\frac{\delta}{4+\delta}}} \\
        &\hspace{15mm} + n \sum_{3d \leq i,k \leq n} \brackets{(\beta_{\floor{\frac{i}{3}}})^{\frac{2+\delta}{4+\delta}} + (a_{\floor{\frac{i}{3}}})^{\frac{2+\delta}{3+\delta}}} \brackets{(\beta_{\floor{\frac{k}{3}}})^{\frac{2+\delta}{4+\delta}} + (a_{\floor{\frac{k}{3}}})^{\frac{2+\delta}{3+\delta}}} \\
        &\hspace{15mm} + n \sum_{3d \leq k \leq n} \brackets{(\beta_{\floor{\frac{k}{3}}})^{\frac{2+\delta}{4+\delta}} + (a_{\floor{\frac{k}{3}}})^{\frac{2+\delta}{3+\delta}}}\Biggr).
    \end{align*}
    Due to the flooring and the summability condition it holds
    \begin{align*}
        \sum_{\substack{3d \leq j \leq n \\ 0 \leq i,k \leq j}} \brackets{(a_{\floor{\frac{j}{3}}})^{\frac{\delta}{3+\delta}} + (\beta_{\floor{\frac{j}{3}}})^{\frac{\delta}{4+\delta}}}
        &= \sum^n_{j=3d} \sum^j_{i,k=0} \brackets{(a_{\floor{\frac{j}{3}}})^{\frac{\delta}{3+\delta}} + (\beta_{\floor{\frac{j}{3}}})^{\frac{\delta}{4+\delta}}} \\
        &= \sum^n_{j=3d} (j+1)^2 \brackets{(a_{\floor{\frac{j}{3}}})^{\frac{\delta}{3+\delta}} + (\beta_{\floor{\frac{j}{3}}})^{\frac{\delta}{4+\delta}}} \\
        &\leq 3 \sum^{n/3}_{j=d} (3j+1)^2 \Bigl(a_j^{\frac{\delta}{3+\delta}} + \beta_j^{\frac{\delta}{4+\delta}}\Bigr) \\
        &\leq 3 \sum^{\infty}_{j=d} 9(j+1)^2 \Bigl(a_j^{\frac{\delta}{3+\delta}} + \beta_j^{\frac{\delta}{4+\delta}}\Bigr) < \infty.
    \end{align*}
    In a similar manner we obtain
    \begin{align*}
        &\sum^n_{i,k=3d} \brackets{(\beta_{\floor{\frac{i}{3}}})^{\frac{2+\delta}{4+\delta}} + (a_{\floor{\frac{i}{3}}})^{\frac{2+\delta}{3+\delta}}} \brackets{(\beta_{\floor{\frac{k}{3}}})^{\frac{2+\delta}{4+\delta}} + (a_{\floor{\frac{k}{3}}})^{\frac{2+\delta}{3+\delta}}} \\
        &\hspace{15mm} = \brackets{\sum^n_{j=3d} (\beta_{\floor{\frac{j}{3}}})^{\frac{2+\delta}{4+\delta}} + (a_{\floor{\frac{j}{3}}})^{\frac{2+\delta}{3+\delta}}}^2 \\
        &\hspace{15mm} \leq 9 \brackets{\sum^{\infty}_{j=d} \beta_j^{\frac{2+\delta}{4+\delta}} + a_j^{\frac{2+\delta}{3+\delta}}}^2 < \infty.
    \end{align*}
    and $\sum_{3d \leq k \leq n} \brackets{(\beta_{\floor{\frac{k}{3}}})^{\frac{2+\delta}{4+\delta}} + (a_{\floor{\frac{k}{3}}})^{\frac{2+\delta}{3+\delta}}} < \infty$, which concludes the proof.
\end{proof}

The central limit theorem for partial sums of functionals of absolutely regular time series given by \cite{borovkovaetal} can be directly extended with regard to our slightly weaker summability condition. The respective proofs are very much the same: First of all, the introduced $(K_n+2L_n,N_n)$-blocking which has to satisfy conditions (3.5)-(3.8) of \cite{borovkovaetal} can be chosen such that $L_n \geq d$. Then, using Lemma \ref{lemma: Lemma 2.23} and \ref{lemma: Lemma 2.24} instead of Lemma 2.23 and 2.24 of \cite{borovkovaetal}, respectively, already yields the desired result:

\begin{theorem}\label{theorem: Theorem 4}
    Let $(X_t)_{t\in\bbn}$ be a 1-approximating functional with approximating constants $(a_k)_{k \in \bbn_0}$ of an absolutely regular time series with mixing coefficients $(\beta_k)_{k\in\bbn_0}$. Furthermore, suppose that $\bbe X_0 = 0$, $\bbe \abs{X_0}^{4+\delta} < \infty$ and
    \begin{equation*}
        \sum^{\infty}_{k=d} k^2 \brackets{a_k^{\frac{\delta}{3+\delta}} + \beta_k^{\frac{\delta}{4+\delta}}} < \infty
    \end{equation*}
    for some $\delta > 0$ and a fixed integer $d \geq 0$. Then, as $n \to \infty$,
    \begin{equation*}
        \frac{1}{\sqrt{n}} \sum^n_{t=1} X_t \xrightarrow{d} N(0, \sigma^2),
    \end{equation*}
    where $\sigma^2 = \bbe X_0^2 + 2 \sum^{\infty}_{k=1} \bbe (X_0X_k)$. In case $\sigma^2=0$, we adopt the convention that $N(0,0)$ denotes the point mass at the origin.
\end{theorem}

\subsection{Law of large numbers and central limit theorem for U-statistics}

Now we want to establish generalizations of the law of large numbers and the central limit theorem of U-statistics of $r$-approximating functionals given by Theorem~6 and Theorem~7 of \cite{borovkovaetal}, respectively.
Utilizing the proof given in that paper though using our Theorem~\ref{theorem: Theorem 3} instead of their Theorem~3, we already obtain the generalization of the law of large numbers:

\begin{theorem}\label{theorem: LLN U-statistics r-approx} \label{theorem: lln u-statistics weaker summability}
    Let $(X_t)_{t\in\bbn}$ be a 1-approximating functional with approximating constants $(a_k)_{k \in \bbn_0}$ of an absolutely regular time series, where $\sum^{\infty}_{k=d} a_k < \infty$ for some integer $d \geq 0$. Furthermore, suppose that $h : \bbr^{2d} \to \bbr^d$ is a measurable and symmetric function which is $1$-continuous, and that the family of random variables $\{h(X_s, X_t) : s,t \geq 1\}$ is uniformly integrable. Then, 
    \begin{equation*}
        U_n(h) = \frac{1}{n(n-1)} \sum_{1 \leq s \neq t \leq n} h(X_s, X_t) \xrightarrow{P} \int_{\bbr^{2d}} h(x, y) dF(x)dF(y) =: \theta
    \end{equation*}
    as $n \to \infty$.
\end{theorem}

What remains to show in order to complete our generalized theory on limit theorems for U-statistics of $r$-approximating functionals is a central limit theorem:

\begin{theorem}\label{theorem: CLT U-statistics r-approx} \label{theorem: clt u-statistics weaker summability}
    Let $(X_t)_{t\in\bbn}$ be a 1-approximating functional with approximating constants $(a_k)_{k \in \bbn_0}$ of an absolutely regular time series with mixing coefficients $(\beta_k)_{k \in \bbn_0}$, and let $h$ be a bounded and 1-continuous kernel. Suppose that the sequences $(a_k)_{k \in \bbn_0}$, $(\beta_k)_{k \in \bbn_0}$ and $(\phi(a_k))_{k \in \bbn_0}$ satisfy the summability condition
    \begin{equation}
        \sum^{\infty}_{k=d} k^2 (\beta_k + a_k + \phi(a_k)) < \infty \label{eq: summability cond}
    \end{equation}
    for some fixed integer $d \geq 0$, where $\phi$ denotes the map used in Definition~\ref{definition: p-lipschitz cond}. Then the series
    \begin{equation}
        \sigma^2 = \Var(h_1(X_0)) + 2 \sum^{\infty}_{k=1} \Cov(h_1(X_0), h_1(X_k)) \label{eq: CLT U-statistics r-approx limit variance}
    \end{equation}
    converges absolutely and, as $n \to \infty$,
    \begin{equation*}
        \sqrt{n} (U_n(h) - \theta) \xrightarrow{d} N(0, 4\sigma^2).
    \end{equation*}
\end{theorem}

The proof is again very similar to the one given by \citet{borovkovaetal}. The main idea is to make use of the Hoeffding decomposition 
\begin{equation*}
\label{eq: Hoeffding decomposition}
    U_n(h) = \theta + \frac{2}{n} \sum^n_{t=1} (h_1(X_t) - \theta) + \frac{2}{n(n-1)} \sum_{1 \leq t_1 < t_2 \leq n} J(X_{t_1}, X_{t_2}),
\end{equation*}
where $\theta = \bbe h(X_1, X_2)$, $h_1(x_1) = \bbe h(x_1, X_2)$ and $J(x_1, x_2) = h(x_1, x_2) - h_1(x_1) - h_1(x_2) + \theta$.
Then we use Theorem~\ref{theorem: Theorem 4} (instead of \citealt[Theorem~4]{borovkovaetal}) to show the convergence of the linear part, that is
\begin{equation*}
    \frac{2}{\sqrt{n}} \sum^n_{t=1} (h_1(X_t)-\theta) \xrightarrow{d} N(0, 4\sigma^2).
\end{equation*}
Let us emphasize at this point that $h_1(X_t)-\theta$ is bounded, since $h$ is a bounded kernel. Hence, the moment assumption in Theorem~\ref{theorem: Theorem 4} is satisfied. We do not use Theorem~\ref{theorem: Theorem 4} on $(X_t)_{t\in \bbn}$ itself, therefore we do not need any additional assumptions on the original time series. For more details on this part of the proof, we refer the reader to \cite[p.~4301]{borovkovaetal}.

Defining 
\begin{equation*}
    R_n := \frac{2}{n(n-1)} \sum_{1 \leq i < j \leq n} J(X_i, X_j)
\end{equation*}
as the degenerate part of the Hoeffding decomposition, it then remains to show that $\sqrt{n} R_n$ is asymptotically negligible under our altered summability condition instead of using \cite[Lemma~4.4]{borovkovaetal}:

\begin{lemma}
    Under the conditions of Theorem \ref{theorem: CLT U-statistics r-approx}, it holds
    \begin{equation}
        \sup_n \bbe \brackets{\frac{1}{n} \sum_{1 \leq i < j \leq n} J(X_i, X_j)}^2 < \infty, \label{eq: lemma 4.4 result}
    \end{equation}
    and therefore, $\sqrt{n} R_n \xrightarrow{P} 0$ as $n \to \infty$.
\end{lemma}

In comparison to the analogous result in \cite[Lemma 4.4]{borovkovaetal}, we need to ensure convergence of series for which we cannot find a bound using our summability condition. Nevertheless, our proof is still similar to the one given by the authors.

\begin{proof}
    For $k \in \{i_1, i_2, j_1, j_2\}$, let $\bbe_{X_k} (J(X_{i_1}, X_{j_1}) J(X_{i_2}, X_{j_2}))$ denote the expected value of $J(X_{i_1}, X_{j_1}) J(X_{i_2}, X_{j_2})$ taken with respect to the random variable $X_k$, with the remaining variables kept fixed. (We adopt this notation from \citealp{borovkovaetal}.)
    Since $J(x,y)$ is a degenerate kernel, i.e., $\int J(x,y) dF(x) = 0$ for all $y \in \bbr^d$ where $F$ denotes the cumulative distribution function of the margins $X_t$, it follows
    \begin{equation}
        \bbe_{X_{j_2}} (J(X_{i_1}, X_{j_1}) J(X_{i_2}, X_{j_2})) = \int_{\bbr^d} J(X_{i_1}, X_{j_1}) J(X_{i_2}, y) dF(y) = 0,
        \label{eq: lemma 4.3 zero}
    \end{equation}    
    and similarly for $\bbe_{X_{i_1}}$, $\bbe_{X_{i_2}}$ and $\bbe_{X_{j_1}}$. As both $h$ and $h_1$ are bounded by definition and 1-continuous by \cite[Lemma~2.15]{borovkovaetal}, the same holds for $g(x_1, x_2, x_3, x_4) = J(x_1, x_2) J(x_3, x_4)$ due to \cite[Lemma~2.14]{borovkovaetal}. Hence, there is a constant $C_1 > 0$ such that both $\bbe \brackets{J(X_{i_1}, X_{j_1}) J(X_{i_2}, X_{j_2})}$ and $\bbe_{X_{i_1}, X_{i_2}} \bbe_{X_{j_1}, X_{j_2}} (J(X_{i_1}, X_{j_1}) J(X_{i_2}, X_{j_2}))$ are bounded by it, that is,
    \begin{align}
        \bbe \brackets{J(X_{i_1}, X_{j_1}) J(X_{i_2}, X_{j_2})} &< C_1, \label{eq: lemma 4.3 bound 1} \\
        \bbe_{X_{i_1}, X_{i_2}} \bbe_{X_{j_1}, X_{j_2}} (J(X_{i_1}, X_{j_1}) J(X_{i_2}, X_{j_2})) &< C_1. \label{eq: lemma 4.3 bound 2}
    \end{align}
    
    By linearity of the expected value it holds
    \begin{align}
        \bbe \biggl(\sum_{1 \leq i < j \leq n} J(X_i, X_j)\biggr)^2 
        &= \bbe \Biggl(\sum_{\substack{1 \leq i_1 < j_1 \leq n \\ 1 \leq i_2 < j_2 \leq n}} J(X_{i_1}, X_{j_1}) J(X_{i_2}, X_{j_2})\Biggr) \nonumber \\
        &= \sum_{\substack{1 \leq i_1 < j_1 \leq n \\ 1 \leq i_2 < j_2 \leq n \\ i_1 = i_2 \text{ and } j_1 = j_2}} \bbe \brackets{J(X_{i_1}, X_{j_1}) J(X_{i_2}, X_{j_2})} \nonumber \\
        &\hspace{15mm} + \sum_{\substack{1 \leq i_1 < j_1 \leq n \\ 1 \leq i_2 < j_2 \leq n \\ i_1 \neq i_2 \text{ or } j_1 \neq j_2}} \bbe \brackets{J(X_{i_1}, X_{j_1}) J(X_{i_2}, X_{j_2})}. \label{eq: lemma 4.4 sum}
    \end{align}
   The first sum on the right-hand side has at most $n(n-1)$ summands, so 
    \begin{equation}
        \sum_{\mathclap{\substack{1 \leq i_1 < j_1 \leq n \\ 1 \leq i_2 < j_2 \leq n \\ i_1 = i_2 \text{ and } j_1 = j_2}}} \bbe \brackets{J(X_{i_1}, X_{j_1}) J(X_{i_2}, X_{j_2})} = \sum_{\mathclap{1 \leq i < j \leq n }} \bbe \brackets{J(X_{i}, X_{j}) J(X_{i}, X_{j})} \leq n(n-1) C_1 \leq n^2 C_1. \label{eq: lemma 4.4 first sum}
    \end{equation}
    In the remainder, we consider the sum given by \eqref{eq: lemma 4.4 sum}. First, let us assume that at least one index is different from all the others, e.g., $j_2$, and suppose that $i_1 \leq i_2 \leq j_1 < j_2$. Let $\Delta_i$ denote the $i$-th largest difference between two consecutive indices. In contrast to \cite{borovkovaetal} we need to split the sum according to some values of $\Delta_i$ and find a different bound for it for small values of $\Delta_i$. If $\Delta_1 = j_2 - j_1$, then 
    \begin{align}
        &\sum_{\substack{1 \leq i_1 \leq i_2 \leq j_1 < j_2 \leq n \\ \Delta_1 = j_2-j_1}} \bbe \brackets{J(X_{i_1}, X_{j_1}) J(X_{i_2}, X_{j_2})} \nonumber \\
        &\hspace{5mm} = \sum_{\substack{1 \leq i_1 \leq i_2 \leq j_1 < j_2 \leq n \\ \Delta_1 = j_2-j_1 < 3d}} \bbe \brackets{J(X_{i_1}, X_{j_1}) J(X_{i_2}, X_{j_2})} \ + \sum_{\substack{1 \leq i_1 \leq i_2 \leq j_1 < j_2 \leq n \\ \Delta_1 = j_2-j_1 \geq 3d}} \bbe \brackets{J(X_{i_1}, X_{j_1}) J(X_{i_2}, X_{j_2})}. \label{eq: lemma 4.4 sum cases}
        \end{align}
        With regard to the first sum appearing in \eqref{eq: lemma 4.4 sum cases}, there are $3d - 1$ values possible for $\Delta_1$ and each of them can be obtained less than $n$ times by shifting $j_1$ and $j_2 = j_1 + \Delta_1$. Furthermore, for fixed $j_1$, at most $3d$ choices for $i_2$ are possible by definition of $0 \leq \Delta_3 \leq \Delta_2 \leq \Delta_1$. The same holds true for $i_1$ and fixed $i_2$, which leads to
        \begin{equation*}
            \sum_{\substack{1 \leq i_1 \leq i_2 \leq j_1 < j_2 \leq n \\ \Delta_1 = j_2-j_1 < 3d}} \bbe \brackets{J(X_{i_1}, X_{j_1}) J(X_{i_2}, X_{j_2})} 
            \leq n (3d)^3 C_1. 
        \end{equation*}
        Now we consider second sum appearing in \eqref{eq: lemma 4.4 sum cases}. Eq.~\eqref{eq: lemma 4.3 zero} yields
        \begin{align*}
            \bbe \brackets{J(X_{i_1}, X_{j_1}) J(X_{i_2}, X_{j_2})} 
            &= \bbe \brackets{J(X_{i_1}, X_{j_1}) J(X_{i_2}, X_{j_2})} \\
            &\hspace{15mm} - \bbe_{X_{i_1}, X_{j_1}, X_{i_2}} \bbe_{X_{j_2}} (J(X_{i_1}, X_{j_1}) J(X_{i_2}, X_{j_2})),
        \end{align*}
        so application of Lemma~4.3 of \cite{borovkovaetal} (with $r=\infty$, $s=1$ and $M=C_1$) yields
        \begin{align*}
        &\sum_{\substack{1 \leq i_1 \leq i_2 \leq j_1 < j_2 \leq n \\ \Delta_1 = j_2-j_1 \geq 3d}} \bbe \brackets{J(X_{i_1}, X_{j_1}) J(X_{i_2}, X_{j_2})} \\
        &\hspace{15mm} \leq 4 C_1  \sum_{\substack{1 \leq i_1 \leq i_2 \leq j_1 < j_2 \leq n \\ \Delta_1 = j_2-j_1\geq 3d}} \brackets{\beta_{\floor{\Delta_1/3}} + a_{\floor{\Delta_1/3}} + 2 \phi\brackets{a_{\floor{\Delta_1/3}}}} \\
        &\hspace{15mm} = 4 C_1 \sum_{\substack{1 \leq j_1 < j_2 \leq n \\ \Delta_1 = j_2-j_1\geq 3d}} (\Delta_1+1)^2 \brackets{\beta_{\floor{\Delta_1/3}} + a_{\floor{\Delta_1/3}} + 2 \phi\brackets{a_{\floor{\Delta_1/3}}}}, 
        \end{align*}
        since there are $\Delta_1+1$ possibilities for $i_2$ for each $j_1$ due to $0 \leq j_1 - i_2 \leq \Delta_1$ by assumption and, dependent on the choice of $i_2$, the same holds true for $i_1$, which results in a total of $(\Delta_1+1)^2$ possibilities for $i_1$ and $i_2$ for each $j_1$. Furthermore, for each $\Delta_1 \geq 3d$ there are $n-\Delta_1$ possible choices for pairs $(j_1, j_2)$ such that $1 \leq j_1 < j_2 \leq n$ and $\Delta_1 = j_2 - j_1$, which yields
        \begin{align*}
        & 4 C_1 \sum_{\substack{1 \leq j_1 < j_2 \leq n \\ \Delta_1 = j_2-j_1\geq 3d}} (\Delta_1+1)^2 \brackets{\beta_{\floor{\Delta_1/3}} + a_{\floor{\Delta_1/3}} + 2 \phi\brackets{a_{\floor{\Delta_1/3}}}} \\
        &\hspace{15mm} \leq 4 C_1 n \sum^{n-3d}_{\Delta_1=3d} (\Delta_1+1)^2 \brackets{\beta_{\floor{\Delta_1/3}} + a_{\floor{\Delta_1/3}} + 2 \phi\brackets{a_{\floor{\Delta_1/3}}}} \\
        &\hspace{15mm} \leq 4 C_1 n \sum^{\floor{\frac{n-3d}{3}}}_{k=d} (3k+1)^2 \brackets{\beta_k + a_k + 2 \phi\brackets{a_k}} \\
        &\hspace{15mm} \leq 36 C_1 n \sum^n_{k=d} (k+1)^2 \brackets{\beta_k + a_k + 2 \phi\brackets{a_k}}.
    \end{align*}
    If $\Delta_1 \neq j_2-j_1$, then, depending on whether $\Delta_1$ is located between $i_1$ and $i_2$ or $i_2$ and $j_1$, by applying Lemma~4.3 of \cite{borovkovaetal} twice, we obtain a bound for $\bbe(J(X_{i_1}, X_{j_1}) J(X_{i_2}, X_{j_2}))$. E.g., if $\Delta_1 = j_1-i_2$, then 
    \begin{align}
        & \bbe(J(X_{i_1}, X_{j_1}) J(X_{i_2}, X_{j_2})) \nonumber \\
        &\hspace{15mm} = \bbe(J(X_{i_1}, X_{j_1}) J(X_{i_2}, X_{j_2})) - \bbe_{X_{i_1}, X_{i_2}} \bbe_{X_{j_1}, X_{j_2}} (J(X_{i_1}, X_{j_1}) J(X_{i_2}, X_{j_2})) \nonumber \\
        &\hspace{30mm} + \bbe_{X_{i_1}, X_{i_2}} \bbe_{X_{j_1}, X_{j_2}} (J(X_{i_1}, X_{j_1}) J(X_{i_2}, X_{j_2})) \nonumber \\
        &\hspace{15mm} \leq 4C_1 \bigl(\beta_{\floor{\frac{\Delta_1}{3}}} + a_{\floor{\frac{\Delta_1}{3}}}\bigr) + 2 \phi\bigl(a_{\floor{\frac{\Delta_1}{3}}}\bigr) + 4C_1 \bigl(\beta_{\floor{\frac{\Delta_2}{3}}} + a_{\floor{\frac{\Delta_2}{3}}}\bigr) + 2 \phi\bigl(a_{\floor{\frac{\Delta_2}{3}}}\bigr), \label{eq: lemma 4.3 bound 3}
    \end{align}
    where we have used \eqref{eq: lemma 4.3 zero} as well as \eqref{eq: lemma 4.3 bound 1} and \eqref{eq: lemma 4.3 bound 2}.
    Note that due to Fubini's theorem, we can always change the order of the expected values in Lemma~4.3 of \cite{borovkovaetal} such that we always obtain the result with respect to $\Delta_2$. Furthermore, note that $\Delta_2 \geq 3d$ implies $\Delta_1 \geq 3d$. \\
    Now we consider \eqref{eq: lemma 4.4 sum} again, but for the case $j_2-j_1 \neq \Delta_1$. In this case it holds 
    \begin{align*}
        &\sum_{\substack{1 \leq i_1 \leq i_2 \leq j_1 < j_2 \leq n \\ j_2-j_1 \neq \Delta_1}} \bbe \brackets{J(X_{i_1}, X_{j_1}) J(X_{i_2}, X_{j_2})} \\
        &\hspace{5mm} = \sum_{\substack{1 \leq i_1 \leq i_2 \leq j_1 < j_2 \leq n \\ \Delta_1 \neq j_2-j_1 \\ \Delta_2 < 3d}} \bbe \brackets{J(X_{i_1}, X_{j_1}) J(X_{i_2}, X_{j_2})} \ + \sum_{\substack{1 \leq i_1 \leq i_2 \leq j_1 < j_2 \leq n \\ \Delta_1 \neq j_2-j_1 \\ \Delta_2 \geq 3d}} \bbe \brackets{J(X_{i_1}, X_{j_1}) J(X_{i_2}, X_{j_2})}.
    \end{align*}
    The first sum is bounded by 
    \begin{equation*}
        \sum_{\substack{1 \leq i_1 \leq i_2 \leq j_1 < j_2 \leq n \\ \Delta_1 \neq j_2-j_1 \\ \Delta_2 < 3d}} \bbe \brackets{J(X_{i_1}, X_{j_1}) J(X_{i_2}, X_{j_2})} 
        \leq \sum_{\substack{1 \leq i_1 \leq i_2 \leq j_1 < j_2 \leq n \\ \Delta_1 \neq j_2-j_1 \\ \Delta_2 < 3d}} C_1
        \leq n^2 (3d)^2 C_1,
    \end{equation*}
    as there are $3d$ possible values for $\Delta_2$, which can be each obtained less than $n$ times, and for each $\Delta_2$ there are $3d$ possible choices for $\Delta_3$. For a given $\Delta_2$, we do not have much information on $\Delta_1$, so we multiply with the factor $n$. On the other hand, by application of \eqref{eq: lemma 4.3 bound 3} for the second sum it holds
    \begin{align*}
        &\sum_{\substack{1 \leq i_1 \leq i_2 \leq j_1 < j_2 \leq n \\ \Delta_1 \neq j_2-j_1 \\ \Delta_2 \geq 3d}} \bbe \brackets{J(X_{i_1}, X_{j_1}) J(X_{i_2}, X_{j_2})} \\
        &\hspace{15mm} \leq 4 C_1 \sum_{\substack{1 \leq i_1 \leq i_2 \leq j_1 < j_2 \leq n \\ \Delta_1 \neq j_2-j_1 \\ \Delta_2 \geq 3d}} \Bigl(\beta_{\floor{\frac{\Delta_1}{3}}} + a_{\floor{\frac{\Delta_1}{3}}} +  \phi\bigl(a_{\floor{\frac{\Delta_1}{3}}}\bigr) 
        + \beta_{\floor{\frac{\Delta_2}{3}}} + a_{\floor{\frac{\Delta_2}{3}}} +  \phi\bigl(a_{\floor{\frac{\Delta_2}{3}}}\bigr)\Bigr) \\
        &\hspace{15mm} \leq 4 C_1 \sum^{n-3d-1}_{\Delta_1=3d} \biggl( n \cdot (\Delta_1 + 1)^2 \cdot \Bigl(\beta_{\floor{\frac{\Delta_1}{3}}} + a_{\floor{\frac{\Delta_1}{3}}} +  \phi\bigl(a_{\floor{\frac{\Delta_1}{3}}}\bigr)\Bigr) \\
        &\hspace{45mm} + \sum^{\Delta_1}_{\Delta_2 =3d} n \cdot (\Delta_2+1) \cdot \Bigl(\beta_{\floor{\frac{\Delta_2}{3}}} + a_{\floor{\frac{\Delta_2}{3}}} +  \phi\bigl(a_{\floor{\frac{\Delta_2}{3}}}\bigr)\Bigr)\biggr) \\
        &\hspace{15mm} \leq 4 C_1 \Biggl( \biggl( \sum^{n}_{\Delta_1=3d} n \cdot (\Delta_1 + 1)^2 \cdot \Bigl(\beta_{\floor{\frac{\Delta_1}{3}}} + a_{\floor{\frac{\Delta_1}{3}}} +  \phi\bigl(a_{\floor{\frac{\Delta_1}{3}}}\bigr)\Bigr)\biggr) \\
        &\hspace{45mm} + \biggl(\sum^{n}_{\Delta_2 =3d} n^2 \cdot (\Delta_2+1) \cdot \Bigl(\beta_{\floor{\frac{\Delta_2}{3}}} + a_{\floor{\frac{\Delta_2}{3}}} +  \phi\bigl(a_{\floor{\frac{\Delta_2}{3}}}\bigr)\Bigr)\biggr)\Biggr) \\
        &\hspace{15mm} \leq 8 C_1 n(n+1) \sum^n_{k=d} (3k+1) \cdot (\beta_k + a_k +  \phi(a_k)),
    \end{align*}
    since $\Delta_1 + 1 \leq n+1$ for $3d \leq \Delta_1 \leq n$. Proceeding in an analogous way in the other cases, by our assumed summability condition \eqref{eq: summability cond} we get that 
    \begin{equation*}
        \sum_{\substack{1 \leq i_1 < j_1 \leq n \\ 1 \leq i_2 < j_2 \leq n \\ i_1 \neq i_2 \text{ or } j_1 \neq j_2}} \bbe \brackets{J(X_{i_1}, X_{j_1}) J(X_{i_2}, X_{j_2})} \leq C_2 n^2,
    \end{equation*}
    for some constant $C_2 > 0$. Combining this with \eqref{eq: lemma 4.4 first sum} we directly obtain \eqref{eq: lemma 4.4 result}.
\end{proof}

\end{document}